\documentclass[12pt,a4paper]{article}
\usepackage[utf8]{inputenc}
\usepackage[T1]{fontenc}
\usepackage{amsmath}
\usepackage{amsfonts}
\usepackage{amssymb}
\usepackage{makeidx}
\usepackage{graphicx}
\usepackage[T1]{fontenc}
\usepackage{csquotes}
\usepackage{hyperref}
\usepackage[french,english]{babel}
\usepackage{amsfonts,amsmath,amssymb}
\usepackage{amsthm}
\usepackage{bbold}
\usepackage{cases}
\usepackage{stackengine}
\usepackage{scalerel}
\usepackage{enumitem}
\usepackage{fullpage} 
\usepackage{graphicx}

\newtheorem{thm}{Theorem}[section]    
\newtheorem{prop}[thm]{Proposition}    
\newtheorem{cor}[thm]{Corollary}    
\newtheorem{lemme}[thm]{Lemma}   
\newtheorem{NB}{\textbf{\underline{Remark}}}

\newcommand\hh{\mathbb{H}}
\newcommand\bbb{\mathbb{B}}
\newcommand{\pr}{\mathbb{P}}
\newcommand{\esp}{\mathbb{E}}

\newcommand{\bbrt}{{B}^{br'}}
\newcommand{\bbr}{B^{br}}

\newcommand{\eps}{\epsilon}
\newcommand{\ph}{\varphi}
\newcommand{\sgn}{\textrm{sign}}
\newcommand{\Bet}{\boldsymbol{\beta}}

\usepackage{authblk}
 \date{}                     

\title{Non co-adapted couplings of {Brownian} motions on {subRiemannian} manifolds}

\author[1]{Magalie Bénéfice}
\affil[1]{{Univ. Bordeaux, CNRS, Bordeaux INP, IMB, UMR 5251},{Talence},
           {F-33400}, 
            {France}}
\begin{document}

\maketitle

\begin{abstract}
In this article we continue the study of couplings of subelliptic Brownian motions on the subRiemannian manifolds $SU(2)$ and $SL(2,\mathbb{R})$. Similar to the case of the Heisenberg group, this subelliptic Brownian motion can be considered as a Brownian motion on the sphere (resp. the hyperbolic plane) together with its swept area modulo $4\pi$. Using this structure, we construct an explicit non co-adapted successful coupling on $SU(2)$ and, under strong conditions on the starting points, on $SL(2,\mathbb{R})$ too. This strategy uses couplings of Brownian bridges, taking inspiration into the work from Banerjee, Gordina and Mariano~\cite{banerjee2017coupling} on the Heisenberg group. We prove that the coupling rate associated to these constructions is exponentially decreasing in time and proportionally to the subRiemannian distance between the starting points. We also give some gradient inequalities that can be deduced from the estimation of this coupling rate.
\end{abstract}

	\section{Introduction}
	\subsection{Motivations}
	In this article we want to construct some non co-adapted successful coupling of subelliptic Brownian motions on $SU(2)$ and $SL(2,\mathbb{R})$ and more precisely study its coupling rate to obtain analysis inequalities. Let first remind the definition of a coupling. A coupling of two probability measures $\mu$ and $\nu$ on $M$ is a probability measure $\pi$ on $M\times M$ such that $\mu$ is its first marginal distribution and $\nu$ its second one. In fact, in the case of two Markov processes $(X_t)_t$ and $(X'_t)_t$, one has to study the joint law of $(X_t,X_t')_t$ to construct a coupling. 
	Coupling probability distributions, in particular Markov chains or Markov processes, is a topic of interest of these late decades as it can offer results not only in Probability and Optimal Transport but also in Analysis and Geometry (see~\cite{lindvall2002lectures} for a general introduction). We interest ourselves in the notion of "coupling time" for a coupling of diffusion processes $(X_t,X_t')_t$. This is the first meeting time of the two processes:
	\begin{equation*}\tau:=\inf\{t>0|X_t=X_t'\}.\end{equation*}
	If the coupling time $\tau$ is a.s. finite, the coupling is called successful. A first interest in the construction of successful couplings, and more precisely in the study of the coupling rate $\pr(\tau>t)$ for $t>0$, has been the estimation of the total variation distance between the laws of $X_t$ and $X_t'$. We recall that this total variation distance is defined by:
\begin{equation*}d_{TV}(\mathcal{L}(X_t),\mathcal{L}(X'_t)):=\sup\limits_{A\ measurable }\{\pr(X_t\in A)-\pr(X'_t\in A)\}.
\end{equation*} 
The Aldous inequality also called Coupling inequality (see~\cite{asmussen2003applied}, chapter VII) states that, for every coupling $(X_s,X_s')_s$ and every $t>0$: \begin{equation}\label{Aldous}\pr(\tau>t)\geq d_{TV}(\mathcal{L}(X_t),\mathcal{L}(X'_t)).
	\end{equation}
With this inequality, one can see the relevance in finding "fast" successful couplings. In particular, couplings that change (\ref{Aldous}) into an equality are called maximal couplings. 
If it has been proved that such couplings always exist in the case of continuous processes on Polish spaces, they can be very difficult to study (simulation, estimation of a coupling rate) as their construction are often non Markovian and even non co-adapted, in the sense that we don't only need the common past of the two processes but also the future of one of these processes to make the construction. In Riemannian manifold, the existence and unicity of Markovian maximal couplings has been discussed in~\cite{ReflKuwada,HsuSturmMaxEuc} for the case of the Brownian motion and in~\cite{BanerjeeMaxElliptic} for elliptic diffusions. The general idea is that the existence of such a coupling needs a sort of "reflection structure" from the Riemannian manifold as well as strong properties from the drift part of the diffusion process. In the case of the Brownian motion such a coupling, if it exists, is the reflection coupling (also called mirror coupling) introduced on $\mathbb{R}^n$ (see~\cite{kendall-coupling-gnl}) and then on Riemannian manifolds (see~\cite{CranstonRefl}).\\
Since the 90's, the study of successful couplings has also led to analytic results about estimates of the spectral gap for elliptic operators (see for example~\cite{WangCouplingEigenvalue,ChenMarkovian,ChenCouplingEigenvalue,ChenWangEigenvalue,MaoCoupling,BaudoinEigenvalue}) or Harnack inequalities (see~\cite{WangHarnack}) in Euclidean spaces as well as in Riemannian manifolds. Note that most of these results use the reflection coupling cited above.\\
Here we place ourselves in subRiemannian manifolds and consider the subelliptic Brownian motions induced by the subLaplacian operator. Our main question is to find how to construct explicit successful couplings. The main difficulty comes from the fact that the subRiemannian distance is, in general, not smooth on every point. Thus the comparison of two Brownian motions become a challenge. A solution was proposed by Baudoin et. al in~\cite{baudoin2022variations} in the case of Sasakian foliations. It consists in approaching the subRiemannian metric by Riemannian ones and using coupling methods from Riemannian structures. Another solution is to use the special structure of some subRiemannian manifolds. A lot of subelliptic, and more generally hypoelliptic diffusions, are written under the form $(X_t,z_t:=f((X_s)_{s\leq t}))_t$ with $(X_t)_t$ an elliptic diffusion on a Riemannian manifold that we will call "the driving noise" and $f$ a functional (see~\cite{baudoin2023stochastic} for some examples). Then, a strategy for coupling such processes consists in coupling the driving noises and see the effects on the "driven processes". It gives numerous examples for couplings, successful or not~\cite{Eberle,Kolmogorov,bonnefont2018couplings,CranstonKolmogorov, kendall-coupling-gnl,kendall2009brownian,kendall2007coupling,banerjee2017coupling,CoAdaptSuccessHeisenberg2, KendallSU(2),banerjee2017coupling,BanerjeeKolmogorov}. This strategy can be used to study the subelliptic Brownian motion $\bbb_t=(X_t,z_t)$ for three model spaces of subRiemannian manifolds:
\begin{itemize}
    \item On the Heisenberg group $\hh$, the driving noise $X_t$ is a $2$-dimensional Brownian motion and $z_t$ is a signed area swept by $(X_s)_{s\leq t}$ in $\mathbb{R}$;
    \item On the special unitary group $SU(2)$, $X_t$ is a Brownian motion on the sphere and $z_t$ is a signed area modulo $4\pi$ swept by $(X_s)_{s\leq t}$;
    \item On the special linar group $SL(2,\mathbb{R})$, $X_t$ is a Brownian motion on the hyperbolic plane and $z_t$ is a signed area modulo $4\pi$ swept by $(X_s)_{s\leq t}$.
\end{itemize}
These three models are in fact fiber bundles with, as a basis, the Riemannian manifold of constant curvature $0$, $1$ and $-1$ respectively (for $SU(2)$, this structure is induced by the well known Hopf fibration). We also have a nice estimation of the subRiemannian distance $d_{cc}$ for this three space models: \begin{equation}\label{equivHeisenberg}
    d_{cc}\left((X_t,z_t),(X_t',z_t')\right)\sim \rho_t+\sqrt{|A_t|}
\end{equation}
with $\rho_t$ the Riemannian distance between $X_t$ and $X_t'$ and $A_t$ a signed swept area between the two driving noises (or its representative modulo $4\pi$ depending of the cases). These geometrical interpretations have been useful to study couplings. In the case of the Heisenberg group, explicit co-adapted (see the works from Kendall and Banerjee~\cite{kendall-coupling-gnl,kendall2009brownian,kendall2007coupling}) and non co-adapted successful couplings (see the work from Banerjee, Gordina and Mariano~\cite{banerjee2017coupling}) have been obtained. In a previous work~\cite{KendallSU(2)}, we extended one of these co-adapted couplings to the case of $SU(2)$. The aim of this article is to continue this work by extending to $SU(2)$ and, in a weaker sense, to $SL(2,\mathbb{R})$, the non co-adapted coupling due to Banerjee et al. Moreover, we interest ourselves in the estimation of the coupling rate. Indeed, in the case of the Heisenberg group, Banerjee et al. obtained a successful coupling starting from any $g,g'\in\hh$ as well as some constant $C$ independent of $g$ and $g'$ such that:
\begin{equation}
    \pr(\tau>t)\leq C \frac{d_{cc}(g,g')}{\sqrt{t}}.
\end{equation}
In particular, if $g$ and $g'$ are on the same fiber, they get  $\pr(\tau>t)\leq C \frac{d_{cc}(g,g')}{t}$, which is a better order than for any co-adapted successful coupling (the order of $\pr(\tau>t)$ is not less than $\frac{1}{\sqrt{t}}$ if the coupling is co-adapted). In this article we give estimates of the coupling rate upon condition of the starting points of the coupling for $SU(2)$ and $SL(2,\mathbb{R})$. We also look at some gradient estimates we can obtain with this coupling method.

\subsection{Results}
Our main result is thus the existence of a non co-adapted coupling with a coupling rate exponentially decreasing and depending of the distance between the starting points of the coupling. Using the decomposition given by the fibration, every element of $SU(2)$ (resp. $SL(2,\mathbb{R})$) can be written on the form $g=(x,z)$ with $x$ an element of the sphere $S^2$ (resp. of the hyperbolic plane $\mathbf{H}^2$) and $z\in]-2\pi,2\pi]$.

	If the starting points of the Brownian motions are in a same fiber, we obtain a successful coupling on $SU(2)$ and on $SL(2,\mathbb{R})$. 
	\begin{thm}\label{Succes2}
		Let $g=(x,z)$, $g'=(x',z')\in SU(2)$ (resp. $SL(2,\mathbb{R})$). We suppose that $x=x'$.\\
		There exists a non co-adapted successful coupling of Brownian motions $(\bbb_t,\bbb_t')$ on $SU(2)$ (resp. $SL(2,\mathbb{R})$) starting at $(g,g')$. Moreover, for all $t_f>0$ and $0<q<1$, there exists $C_q$ and $c$ some non negative constants that do not depend on the starting points of the process, such that, for all $t>t_f$:
			\begin{equation}\label{inegalitéCouplage}
				\pr(\tau>t)\leq \left(C_q e^{-ct}d_{cc}(g,g')^{2q}\right)\wedge 1.
			\end{equation}
	\end{thm}
	 In particular, in the case of $SU(2)$, we can use the previous result to construct a successful coupling whatever the starting points of the processes:
	 \begin{thm}\label{Succes1}
		Let $g$, $g'\in SU(2)$. There exists a non co-adapted successful coupling of Brownian motions $(\bbb_t,\bbb_t')$ on $SU(2)$ starting at $(g,g')$. Moreover, denoting $\tau:=\inf\{t|\bbb_t=\bbb_t'\}$, for all $t_f>0$, there exists $C$, $c$ some non negative constants that do not depend on the starting points of the process, such that, for all $t>t_f$:
		\begin{equation}\label{inegalitédcc1}
			\pr(\tau>t)\leq \left(Ce^{-ct}d_{cc}(g,g')\right)\wedge 1.
		\end{equation}
	\end{thm}
	Please note that, with Theorem \ref{Succes2} and Theorem \ref{Succes1}, we improve and give a proof of the results announced during the GSI'23 Conference (\cite{GSI}, Theorem 3).\\
	The coupling strategy is as follow:
	\begin{enumerate}
	    \item Supposing that $x=x'$, we construct a coupling for $t\in[0,T]$ such that $X_T=X_T'$ a.s. and $\pr(z_T=z_T')>0$. To do that, we use well chosen couplings of Brownian bridges to construct $(X_t,X_t')$. Note that, if this construction takes inspiration from the strategy developed by Banerjee et. al to deal with the Heisenberg group case, it needs here a lot of adaptations due to the presence of the curvature. In particular the coupling that is exposed here is done on the spherical/polar coordinates of the processes whereas the one in~\cite{banerjee2017coupling} was done on Cartesian coordinates. One of the consequences is that, if the quantity $\pr(z_T=z_T')$ is well described for all $T$ for the Heisenberg group, we obtain much less information in the model spaces considered here. We yet get some estimations for $T$ small enough comparable to the results on $\hh$.
	    \item Using the fact that $z_t$ takes its values in the compact $[-2\pi,2\pi]$, we can obtain a positive lower bound for $\pr(z_T= z_T')$ which is independent of the starting points. The iteration of the previous construction can then be compared to the iteration of identically and independently distributed experiments. Thus, we obtain a successful coupling with a coupling rate exponentially decreasing.
	    Moreover, for any $t_f>0$, we obtain:
	    	\begin{equation}
			\pr(\tau>t)\leq C_q e^{-c t}|z-z'|^{q}  \text{ for all }t>t_f
		\end{equation}
		with $T>0$, $C_q$, $c$ independent of the starting points.

In particular, under the above condition $x=x'$, we have the equivalence $d_{cc}(g,g')\sim \sqrt{|z-z'|}$ (see relation (\ref{d_cc}) in section \ref{PremiminariesSubRiemannian}). This leads to inequality (\ref{inegalitéCouplage}).

\item In the case of $SU(2)$, when $x\neq x'$, the idea is to use a successful coupling on the sphere, in fact the reflection coupling, to obtain $X_t=X_t'$ at an almost surely finite time and then, use the strategy given above. Note that this can't be done in the case of $SL(2,\mathbb{R})$ as there exists no successful coupling on the hyperbolic plane (we can state it using Theorem (5.4) from Wang~\cite{wang2002liouville}, as there exists some non constant but bounded harmonic functions on the hyperbolic plane).
\end{enumerate}

\begin{NB}
    Please note that, by taking polar coordinates, this coupling strategy gives a coupling on the Heisenberg group which is different from the one in~\cite{banerjee2017coupling}. However we cannot say that this new coupling will be successful. Indeed, as the swept area is not bounded in the Heisenberg group, to obtain a successful coupling it seems that we would need to iterate our construction for geometrically increasing intervals of time $T$ (see~\cite{banerjee2017coupling} for more details).  As explained before, as we lack information on the quantity $\pr(z_T=z_T')$ for $T$ too large, we cannot make any conclusion. 
\end{NB}
From these theorems, we can deduce some analytic results. We denote by $\nabla_{\mathcal{H}}$ the subgradient induced by the subLaplacian operator and by $||\cdot||_{\mathcal{H}}$ the norm induced by the subRiemannian structure. We get:
\begin{cor}\label{grad}
For every function $f\in\mathcal{C}^2(SU(2))$, we have:
\begin{equation}\label{gradientInequality}
    ||\nabla_{\mathcal{H}}P_tf(g)||_{\mathcal{H}}\leq2||f||_{\infty}Ce^{-ct}\text{ a.e.}
\end{equation}
In particular, if $f$ is harmonic on $SU(2)$, then it is constant.
\end{cor}
We also get some results for $SL(2,\mathbb{R})$:
\begin{cor}\label{gradSL}
Let $g=(x,z)$, $g'=(x',z')\in SL(2,\mathbb{R})$. We suppose that $x=x'$ and we consider a bounded function $f\in\mathcal{C}^2(SL(2,\mathbb{R}))$. For all $q\in]0,1[$, there exists $C_q$, $c$ some constants independent of $g$, $g'$ such that :
\begin{equation*}|P_tf(g)-P_tf(g')|\leq 2||f||_{\infty}C_q e^{-ct}d_{cc}(g,g')^{2q}.\end{equation*}
Moreover, if $f$ is harmonic and bounded, it is constant on each fiber $\{(x,z)\in SL(2,\mathbb{R}) \ | \ z\in[-2\pi,2\pi]\}$ above $x$.
\end{cor}
Note that in~\cite{ArnaudonThalmaier}, Arnaudon and Thalmaier obtained some expressions for $\nabla_{\mathcal{H}}P_tf(g)$ in the cases of $SU(2)$ (Theorem 3.2) but also $SL(2,\mathbb{R})$ (Theorem 7.1). These expressions are obtained in function of an adapted process $\left(\phi_t\right)_t$ leaving in the cotangent bundle at point $g$. In particular, this leads to: 
\begin{equation*}
     ||\nabla_{\mathcal{H}}P_tf(g)||_{\mathcal{H}}\leq||f||_{\infty}\times C(t) \text{ with } C(t)<\infty.
\end{equation*} For the moment, there is no easy estimation of $\left(\phi_t\right)_t$ and thus no easy estimation of $C(t)$ contrary to our result from Corollary \ref{grad} which offers a decreasing in long times.
\subsection{Plan}
The structure of the paper is as follows. In the second section we introduce the space models $SU(2)$ and $SL(2,\mathbb{R})$ and give some results about their subelliptic Brownian motions. In section 3 we provide two lemmas about the subRiemannian distance between the subelleptic Brownian motions and about exit times of real Brownian motions that will be used all along the article. In a fourth section we describe the coupling strategy in $SU(2)$ and $SL(2,\mathbb{R})$ for the case where the starting points of the Brownian motions are in the same fiber which provides the proof of Theorem \ref{Succes2}. The general case for $SU(2)$, using the reflection coupling on the sphere and proving Theorem \ref{Succes1}, is completed in section 5. Finally, the analytic results from Corollaries \ref{grad} and \ref{gradSL} are proven in section 6.
\section{Preliminaries}
	\subsection{SubRiemannian structure}\label{PremiminariesSubRiemannian}
	We consider the two following matrices groups that we will denote $E_k$, with $k\in\{-1,1\}$:
	\begin{itemize}
		\item  $E_1:=SU(2)$  denotes the group of the unitary two dimensional matrices with complex coefficients and with determinant 1;
		\item $E_{-1}:=SL(2,\mathbb{R})$ denotes the group of two dimensional matrices with real coefficients and with determinant $1$.
	\end{itemize}
Considering the manifold structure induced by the usual topology on the matrices group and, as the application $\Big{\{}\begin{array}{ccc}
			E_k\times E_k&\to &E_k\\
			(A,B)&\mapsto &A^{-1}\cdot B
		\end{array}$ is smooth, $E_k$ is a Lie group. We denote by $\mathfrak{e}_k$ the Lie algebra associated to $E_k$. It is canonically identified to the tangent space of $E_k$ at the neutral point $I_2$. To construct the usual SubRiemannian structure, we chose a basis $(X,Y,Z)$ of this algebra such that:
	\begin{equation}\label{LieSU} [X,Y]=Z\ , \ [Y,Z]=kX\  \text{and } [Z,X]=kY.
	\end{equation}
For $SU(2)$ ($k=1$), we take the Pauli matrices:
	 \begin{equation*}X=\frac{1}{2}\begin{pmatrix} 0 &1 \\ -1 &0 \end{pmatrix}, Y=\frac{1}{2}\begin{pmatrix} 0 &i \\ i &0 \end{pmatrix} \text{ and }Z=\frac{1}{2}\begin{pmatrix} i &0 \\ 0 &-i \end{pmatrix}.
	 \end{equation*}
 For $SL(2,\mathbb{R})$ ($k=-1$), we take $X=\frac{1}{2}\begin{pmatrix} 1 &0 \\ 0 &-1 \end{pmatrix}$, $Y=\frac{1}{2}\begin{pmatrix} 0 &-1 \\ -1 &0 \end{pmatrix}$ and $Z=\frac{1}{2}\begin{pmatrix} 0 &-1 \\ 1 &0 \end{pmatrix}$. 
  It is important to notice that every element of $E_k$ can be written on the form \begin{equation}\exp(\ph(\cos(\theta)X+\sin(\theta)Y))\exp(zZ)\end{equation} with:
 	\begin{itemize}
 		\item  for $k=1$: $\ph\in[0,\pi[$,  $\theta\in[0,2\pi[$ and $z\in]-2\pi,2\pi]$;
 		\item for $k=-1$: $\ph\in[0,+\infty[$, $\theta\in[0,2\pi[$ and $z\in]-2\pi,2\pi]$.
 	\end{itemize} Note that $\ph(\cos(\theta)X+\sin(\theta)Y)$ and $z$ are unique. This provides some coordinates $(\ph,\theta,z)$ called cylindrical coordinates. They will be of importance in all the paper. In particular there exists a natural projection $\Pi_k$ from $E_k$ to the Riemannian manifold $M_k$ of constant curvature $k$ (that is the sphere $S^2$ in the case of $SU(2)$ and the hyperbolic plane $\mathbf{H}^2$ in the case of $SL(2,\mathbb{R})$), sending $(\ph,\theta,z)$ to the point of $S^2$ (resp. $\mathbf{H}^2$) described by the spherical coordinates $(\ph,\theta)$ according to the north pole $N_0:=(0,0,1)\in S^2$ and the vector $e_0:=(0,-1,0)\in T_{N_0}S^2$ (resp. the polar coordinates relative to the pole $N_0:=i\in \mathbf{H}^2$ and the vector $e_0:=i\in T_{N_0}\mathbf{H}^2$) (see~\cite{KendallSU(2)} for more details). In the case of $SU(2)$, this projection is the one induced by the Hopf fibration.
 
We can then define a basis on all the tangent space $TE$ by considering $\bar{X}$, $\bar{Y}$ and $\bar{Z}$ the left-invariant vector fields associated to $X$, $Y$, $Z$. For $g\in E_k$ they are given by: \begin{equation*}\bar{A}_g=\frac{\partial}{\partial\eps}_{|\eps=0}\left(g\exp(\eps A)\right)\text{ for }A=X,Y,Z.
\end{equation*}
By considering $\mathcal{H}:=Vect\langle\bar{X},\bar{Y}\rangle$, we define a subspace of the tangent space $TE_k$ that we call horizontal space. The subRiemannian structure is defined by considering curves $\gamma:J\subset \mathbb{R}\to E_k$ that "move" only with directions in $\mathcal{H}$, in the sense that $\dot{\gamma}(t)\in \mathcal{H}_{\gamma(t)}$ for all $t\in J$. Such curves are called horizontal curves. We construct a scalar product $\langle\cdot,\cdot\rangle_{\mathcal{H}_{g}}$ on $\mathcal{H}_{g}$ for all $g\in E_k$, such that $(\bar{X}_g,\bar{Y}_g)$ is an orthonormal basis. The same way as for a Riemannian structure, we obtain a length $L(\gamma)$ of the horizontal curve $\gamma$:
\begin{equation*}
	L(\gamma):=\int_I \sqrt{\langle\dot{\gamma}(t),\dot{\gamma}(t)\rangle_{\mathcal{H}_{\gamma(t)}}}dt.
\end{equation*}
The Carnot-Caratheodory distance between $g$ and $h\in E$ is finally defined by:
\begin{equation*}
	d_{cc}(g,h)=\inf\{ L(\gamma) \ | \ \gamma \text{ horizontal curve between } g \text{ and }h\}.
\end{equation*}
Thanks to relation (\ref{LieSU}), the parabolic Hörmander conditions are satisfied in the sense that $\mathcal{H}$ is Lie-bracket generating. A consequence is that $d_{cc}$ is finite and the subRiemannian structure is well defined. Note that, because we have chosen left invariant vector fields, the Carnot-Caratheodory distance is left invariant too. Moreover, from~\cite{baudoin2009subelliptic,bonnefont-these,bonnefont2012subelliptic} we have:
\begin{equation}\label{d_cc} d_{cc}^2(0,(\ph,\theta,z)) \text{ is equivalent to } \ph^2+|z|.
	\end{equation}
Because we deal with Lie groups and because the Hörmander conditions are satisfied, we can introduce a subelliptic diffusion operator, the subLaplacian operator: \begin{equation*}L=\frac{1}{2}\left(\bar{X}^2+\bar{Y}^2\right).
\end{equation*}
\subsection{Brownian motions on $SU(2)$ and $SL(2,\mathbb{R})$}
	Given this operator, we can define the Brownian motion on these Lie groups as the Markov process $\bbb_t$ with infinitesimal generator $L$. By using the cylindrical coordinates, the continuous Brownian motion $\bbb_t$ can be written $\bbb_t=\exp(\ph_t(\cos(\theta_t)X+\sin(\theta_t)Y))\exp(z_t Z)$ with $\ph_t$, $\theta_t$ and $z_t$ three continuous real diffusion processes satisfying the differential stochastic equations
		\begin{equation}\begin{cases}\label{equation}
			d\ph_t= dB_t^1+\frac{1}{2}\sqrt{k}\cot(\sqrt{k}\ph_t)dt\\
			d\theta_t=\frac{\sqrt{k}}{\sin(\sqrt{k}\ph_t)}dB_t^2\\
			dz_t=\frac{\tan\left(\frac{\sqrt{k}\ph_t}{2}\right)}{\sqrt{k}}dB_t^2
		\end{cases}
	\end{equation}
where $B_t^1$ and $B_t^2$ are two independent real Brownian motions. In particular, we get from~\cite{baudoin2009subelliptic,bonnefont2012subelliptic,bonnefont-these} the following geometric interpretation:
\begin{prop}[\cite{baudoin2009subelliptic,bonnefont2012subelliptic,bonnefont-these}]\label{interpretation}
	In $SU(2)$ (resp. $SL(2,\mathbb{R}))$, $(\ph_t,\theta_t)$ are the spherical coordinates (resp. polar coordinates) relative to $(N_0,e_0)$ of a Brownian motion on the sphere $S^2$ (resp. the hyperbolic plane $\mathbf{H}^2$) and $z_t-z_0$ is the signed swept area of $(\ph_t,\theta_t)$ with respect to the fixed pole $N_0$.
\end{prop}
We now consider two Brownian motions on $E_k$: $\bbb_t=(\ph_t,\theta_t,z_t)$ and $\bbb'_t=(\ph'_t,\theta'_t,z'_t)$. To compare the two Brownian motions, and in particular to have an estimation of the Carnot-Caratheodory distance, using (\ref{d_cc}), we need to study $\bbb_t^{-1}\cdot\bbb'_t$. From~\cite{KendallSU(2)}, we have:
\begin{prop}[\cite{KendallSU(2)}]\label{equiv}
	Let denote by $X_t=(\ph_t,\theta_t)$ and $X_t'=(\ph_t',\theta_t')$ the projection of the Brownian motions $\bbb_t=\exp(\ph_t(\cos(\theta_t)X+\sin(\theta_t)Y))\exp(z_t Z)$ and $\bbb_t'=\exp(\ph'_t(\cos(\theta'_t)X+\sin(\theta'_t)Y))\exp(z'_t Z)$ on the sphere (resp. hyperbolic plane). The cylindrical coordinates of $\bbb_t^{-1}\cdot\bbb'_t$ are given by $(\rho_t,\Theta_t,\zeta_t)$ with
	\begin{itemize}
		\item $\rho_t$ equal to the usual Riemannian distance between $X_t$ and $X_t'$.
		\item $\zeta_t\equiv z'_t-z_t+\sgn(\theta_t-\theta'_t)\mathcal{A}_{X_t,X'_t,N_0}\mod(4\pi)$
		with $\mathcal{A}_{a,b,c}$ the area of the spherical (resp. hyperbolic) triangle of vertices $a, b$ and $c$ and $N_0$ the pole of reference chosen by the projection $\Pi_k$. 
	\end{itemize}
	In particular, reminding that, by definition, $\zeta_t\in]-2\pi, 2\pi]$,  we have: \begin{equation*}
		d_{cc}^2(\bbb_t,\bbb'_t)\sim \rho_t^2+|\zeta_t|
	\end{equation*} 
\end{prop}
\begin{NB}\label{Heisenberg}
	Similar results are well known (and easier to obtain) in the Heisenberg group. First, the subRiemannian structure is defined such that the relations (\ref{LieSU}) are true for $k=0$. The same way, the Brownian motion can be seen as a Brownian motion $X_t$ in the plane together with its swept area $z_t$ (the Levy's area). In fact when $X_t$ is expressed in the polar coordinates $(\ph_t,\theta_t)$, the Brownian motion exactly satisfies equation (\ref{equation}), taking the limit $0$ for $k$. Propositions (\ref{interpretation}) and (\ref{equiv}) are also true with the use of planar triangles.
\end{NB}

\section{Some useful Lemmas}
\subsection{Stability of $z_t'-z_t$ under changes of coordinates}
In Proposition \ref{equiv}, the quantities $\rho_t$ and $\zeta_t$ are intrinsic to $\bbb_t^{-1}\bbb_t$ and do not depend on the choice of the projection $\Pi_k$. As seen in this same proposition, this does not seem to be the case of the quantity $z_t'-z_t$ as the quantity $sign(\theta_t-\theta'_t)\mathcal{A}_{X_t,X'_t,N_0}\ mod\ (4\pi)$ depends on the choice of the pole $N_0$ and the vector $e_0$. To prove our theorems, we study this non intrinsic quantity. For various reasons we need to change the system of spherical/polar coordinates on $M_k$ induced by $\Pi_k$, that is we change the pole and the vector of reference. Thus it is interesting to see how $z_t'-z_t$ reacts. Let consider $(X_t)_t$ and $(X'_t)_t$ as in Proposition \ref{equiv}. We chose $(N,e)\in TM_k$. Let introduce some notations.\\
\begin{itemize}
        \item We denote by $(\ph^{(N,e)}_t,\theta^{(N,e)}_t)$ (resp. $({\ph'}^{(N,e)}_t,{\theta'}^{(N,e)}_t)$) the spherical/polar coordinates of $X_t$ (resp. $X_t'$) relative to $N$ and $e$.
        \item We denote by $I_t{(N,e)}$ (resp. $I'_t{(N,e)}$) the signed area swept by $(X_s)_{s\leq t}$ (resp. $(X'_s)_{s\leq t}$) relative to $N$ and $e$ and starting at point $z_0$ (resp. $z_0'$). More precisely, it is defined such that $(\ph^{(N,e)}_t,\theta^{(N,e)}_t,I_t{(N,e)})$ satisfies the stochastic differential equations (\ref{equation}). 
    \item We denote by $A_t$ the signed swept area between $(X_s)_{s\leq t}$ and $(X'_s)_{s\leq t}$, that is the area delimited by $(X_s)_{s\leq t}$, $(X'_s)_{s\leq t}$ and the geodesics joining $X_0$ to $X_0'$ and $X_t$ to $X_t'$ with the sign changing when the paths are crossing (see~\cite{KendallSU(2)} for more details). Note that this quantity does not depend on the choice of $N$ and $e$.
\end{itemize} 
In particular, we have $(\ph_t,\theta_t)=(\ph^{(N_0,e_0)}_t,\theta^{(N_0,e_0)}_t)$ and $z_t=I_t{(N_0,e_0)}$. Note that $z_t\neq I_t(N,e)$ in general for $(N,e)\neq (N_0,e_0)$.
Then we have the following results:

\begin{lemme}\label{aireBalayee}
     For all $(N,e)\in TM_k$,we have:
     \begin{align}
	A_t=I'_t{(N,e)}-I_t{(N,e)}-(z'_0-z_0)&+\sgn\left(\theta^{(N,e)}_t-{\theta'}^{(N,e)}_t\right) \mathcal{A}_{X_t,X'_t,N}\notag\\
	&-\sgn\left(\theta^{(N,e)}_0-{\theta'}^{(N,e)}_0\right) \mathcal{A}_{X_0,X'_0,N}.\label{AireBalayée1}
	\end{align}
	In particular, we have:
	\begin{align}\label{AireBalayee2}
	\zeta_t\equiv I'_t({N,e})-I_t({N,e})&+\sgn\left(\theta^{(N,e)}_t-{\theta'}^{(N,e)}_t\right) \mathcal{A}_{X_t,X'_t,N}-\sgn\left(\theta^{(N,e)}_0-{\theta'}^{(N,e)}_0\right) \mathcal{A}_{X_0,X'_0,N}\notag\\
	&
	+\sgn\left(\theta^{(N_0,e_0)}_0-{\theta'}^{(N_0,e_0)}_0\right) \mathcal{A}_{X_0,X'_0,N_0}\mod(4\pi).
\end{align}
\end{lemme}
\begin{proof}
Relation (\ref{AireBalayée1}) is an immediate geometric result. We look at the second relation (\ref{AireBalayee2}). From Proposition \ref{equiv}, we have:
\begin{equation*}
    \zeta_t\equiv I'_t{(N_0,e_0)}-I_t{(N_0,e_0)}+\sgn\left(\theta^{(N_0,e_0)}_t-{\theta'}^{(N_0,e_0)}_t\right)\mathcal{A}_{X_t,X'_t,N_0}\mod(4\pi).
    \end{equation*}
    Then using (\ref{AireBalayée1}) with $N_0$, we get:
    \begin{equation*}
    \zeta_t\equiv z_0'-z_0+A_t+\sgn\left(\theta^{(N_0,e_0)}_0-{\theta'}^{(N_0,e_0)}_0\right)\mathcal{A}_{X_0,X_0',N_0}\mod(4\pi).
\end{equation*}
Using (\ref{AireBalayée1}) this time with $N$, we obtain the expected result.
\end{proof}
\begin{NB}\label{aireBalayee3}

If we have $X_0=X_0'$ and $X_T=X_T'$, then for any $(N,e)\in TM_k$:
 \begin{equation*}
  	\zeta_t\equiv I'_t{(N,e)}-I_t{(N,e)} \mod(4\pi).
\end{equation*}
\end{NB}
\begin{NB} The impact of these change of coordinates can also be seen directly in $E_k$ on the process $\bbb_t$. Let $(N,e)\in TM_k$. We claim that the process: \begin{equation*}J_t:=\exp\left(\ph^{(N,e)}_t\left(\cos(\theta^{(N,e)}_t)X+\sin(\theta^{(N,e)}_t)Y\right)\right)\exp\left(I_t{(N,e)}Z\right)
\end{equation*} can be obtained from $\bbb_t$ by looking at $g^{-1}\bbb_t\exp(zZ)$ for $g$ depending only on the choice $(N,e)$ and $z$ depending on $N$, $e$ but also $X_0$. Let explain this fact. 
Taking $g=(\ph_g,\theta_g,z_g)\in E_k$, it is possible to prove that for all $h\in E_k$, $\Pi_k(g^{-1}h)=m_g\left(\Pi_k(h)\right)$ with $m_g$ a direct isometry in $M_k$ that can be decomposed as follow:
         \begin{itemize}
             \item we first make a rotation of angle $-z_g$ and of axis directed by $(0,0,1)$ for $SU(2)$ (resp. of center $N_0$ for $SL(2,\mathbb{R})$). It keeps the pole $N_0$ invariant but acts on the vector of reference $e_0=T_{I_2}\Pi_k(X)$;
             \item we then apply a direct isometry which acts by translation on the geodesic from $\Pi_k(g)$ to $N_0$. In particular this isometry transports the new vector obtained above parallelly along this geodesic. The vector obtained finally is equal to $T_g\Pi_k(\bar{X})$.
         \end{itemize}
         Thus, for $(N,e)\in TM_k$, we can find $g\in E_k$ such that $\Pi_k(g)=N$ and $T_g\Pi_k(\bar{X})=e$. For all $t$, $\Pi_k(g^{-1}\bbb_t)$ gives the polar coordinates of $\Pi_k(\bbb_t)$ relative to  $N$ and $e$, that is $\left(\ph^{(N,e)}_t,\theta^{(N,e)}_t\right)$. 
         With similar results as in Proposition \ref{equiv}, we have $g^{-1}\bbb_t$ equal to: \begin{align*}\label{J_t aire}
         \exp&\left(\ph^{(N,e)}_t\left(\cos(\theta^{(N,e)}_t)X+\sin(\theta^{(N,e)}_t)Y\right)\right)\\
         &\times\exp\left(\left(z_t-z_g+\sgn\left(\theta_g-\theta_t^{(N_0,e_0)}\right)\mathcal{A}_{N,X_t,N_0}\right)Z\right).
         \end{align*}
         Using some geometric comparisons, we can obtain: \begin{equation*}
             z_t-z_g+\sgn\left(\theta_g-\theta_t^{(N_0,e_0)}\right)\mathcal{A}_{N,X_t,N_0}
             =I_t{(N,e)}-z_g+\sgn\left(\theta_g-\theta_0^{(N_0,e_0)}\right)\mathcal{A}_{N,X_0,N_0}.
         \end{equation*}
         The process $J_t$ is then equal to: 
         \begin{equation*}
              J_t=g^{-1}\bbb_t\exp\left(\left(z_g-\sgn\left(\theta_g-\theta_0^{(N_0,e_0)}\right)\mathcal{A}_{N,X_0,N_0}\right)Z\right)
         \end{equation*}
        
         We now consider this change of coordinates for the two processes $\bbb_t$ and $\bbb_t'$ that we want to compare. As before, we define 
         \begin{equation*}J_t':=\exp\left({\ph'}_t^{(N,e)}\left(\cos({\theta'_t}^{(N,e)})X+\sin({\theta'_t}^{(N,e)})Y\right)\right)\exp\left(I'_t{(N,e)}Z\right).
\end{equation*} 
         We have \begin{equation*}
              J_t'=g^{-1}\bbb'_t\exp\left(\left(z_g-\sgn\left(\theta_g-{\theta_0'}^{(N_0,e_0)}\right)\mathcal{A}_{N,X'_0,N_0}\right)Z\right)
         \end{equation*}
         Here it seems evident that, in general $d_{cc}(\bbb_t,\bbb_t')\neq d_{cc}(J_t,J_t')$. However, if $X_0=X_0'$ and $X_T=X_T'$, as $\exp(\alpha Z)$ and $\exp(\beta Z)$ commute for all $\alpha$, $\beta\in\mathbb{R}$, we have $\bbb_T^{-1}\bbb_T'=\exp(-z_T)\exp(z_T')$ and:
         \begin{align*}
         J_T^{-1}J_T'=\exp&\left(\left(-z_g+\sgn\left(\theta_g-\theta_0^{(N_0,e_0)}\right)\mathcal{A}_{N,X_0,N_0}\right)Z\right)\exp(-z_T)\exp(z_T')\\
         &\exp\left(\left(z_g-\sgn\left(\theta_g-{\theta_0'}^{(N_0,e_0)}\right)\mathcal{A}_{N,X_0,N_0}\right)Z\right)=\exp(-z_T)\exp(z_T').
         \end{align*}
         Thus at time $T$, $d_{cc}(\bbb_T,\bbb_T')=d_{cc}(J_T,J_T')$. This gives an echo of Remark~\ref{aireBalayee3}.\\
         In general, by left invariance of the Carnot Carathéodory distance, we have $d_{cc}(\bbb_t,\bbb_t')=d_{cc}(g^{-1}\bbb_t,g^{-1}\bbb_t')$. In fact, $\left(g^{-1}\bbb_t\right)^{-1}g^{-1}\bbb_t'=\bbb_t^ {-1}\bbb_t$ and thus, the third cylindrical coordinate of $\left(g^{-1}\bbb_t\right)^{-1}g^{-1}\bbb_t'$ is equal to $\zeta_t$ as defined in Proposition~\ref{equiv}. Applying Proposition \ref{equiv} on $g^{-1}\bbb_t'$ and $g^{-1}\bbb_t$, we obtain the following equality modulo $4\pi$:
         \begin{align*}
            \zeta_t=I_t'{(N,e)}-I_t{(N,e)}+\sgn\left(\theta_g-{\theta'_0}^{(N_0,e_0)}\right)&\mathcal{A}_{N,X'_0,N_0}-\sgn\left(\theta_g-\theta_0^{(N_0,e_0)}\right)\mathcal{A}_{N,X_0,N_0}\\
            &+\sgn\left(\theta_t^{(N,e)}-{\theta'_t}^{(N,e)}\right)\mathcal{A}_{m_g(X'_t),m_g(X_t),N_0}\\
         \end{align*}
        As we have:
        \begin{align*}
        \sgn&\left(\theta_g-{\theta'_0}^{(N_0,e_0)}\right)\mathcal{A}_{N,X'_0,N_0}-\sgn\left(\theta_g-\theta_0^{(N_0,e_0)}\right)\mathcal{A}_{N,X_0,N_0}\\
       & =-\sgn\left(\theta_0^{(N_0,e_0)}-{\theta_0'}^{(N_0,e_0)}\right)\mathcal{A}_{X_0,X'_0,N}+\sgn\left(\theta_0^{(N_0,e_0)}-{\theta_0'}^{(N_0,e_0)}\right)\mathcal{A}_{X_0,X'_0,N_0},
        \end{align*} 
        and $\mathcal{A}_{m_g(X'_t),m_g(X_t),N_0}=\mathcal{A}_{X'_t,X_t,m_g^{-1}(N_0)}$ with $m_g^{-1}(N_0)=N$,
        we obtain the relation (\ref{AireBalayee2}) from Lemma \ref{aireBalayee}.
         \end{NB}
\subsection{First hitting time, first exit time for one-dimensional Brownian motions}
Let $W$ be a one dimensional Brownian motion starting at $0$. The results we give here are well known and can be found in numerous references. 
They will be used later to obtain estimates of the coupling rates. We first begin with relations about first hitting time.
\begin{lemme}\label{hittingTime}
Let $a\in\mathbb{R}$. We denote $D_a:=\inf\{t >0\ | \ W_t=a\}$, the first hitting time of $a$ by $W$. We have, for all $t>0$:
\begin{equation*}
    \pr(D_a>t)\leq \left(\frac{|a|}{\sqrt{2\pi t}}\right)\wedge 1.
\end{equation*}
\end{lemme}
\begin{proof}
 The density of $D_a$ is well known, given by $g_a(u)=\frac{|a|}{\sqrt{2\pi u^3}}\exp(-\frac{a^2}{2u})\times \mathbb{1}_{[0,+\infty[}(u)$. We just have to upper bound the exponential part in this density to obtain the above inequality.
\end{proof}
We now list some relations involving the first exit time of a Brownian motion from an open set:
\begin{lemme}\label{tpsSortie}
 We set two reals $a$ and $b$ such that $a<0<b$ and $H_{a,b}=\inf\{t>0| W_t\notin]a,b[\}$.
	Then, for $\delta>0$, we get:
	\begin{align}
		\esp[e^{-\delta H_{a,b}}]&=\frac{\cosh\left(\sqrt{\frac{\delta}{2}}(a+b)\right)}{\cosh\left(\sqrt{\frac{\delta}{2}}(b-a)\right)}; \label{H1} \\
		\esp[e^{\delta H_{a,b}}]&=\frac{\cos\left(\sqrt{\frac{\delta}{2}}(a+b)\right)}{\cos\left(\sqrt{\frac{\delta}{2}}(b-a)\right)} \text{ if } \sqrt{\frac{\delta}{2}}(b-a)\in]0,\frac{\pi}{2}[; \label{H2} \\
		\esp[H_{a,b} e^{\delta H_{a,b}}]&
		\leq \frac{-ab}{\cos^2\left(\sqrt{\frac{\delta}{2}}(b-a)\right)} \text{ if } \sqrt{\frac{\delta}{2}}(b-a)\in]0,\frac{\pi}{2}[;\label{H3bis}\\
		\esp[H_{a,b}]&=-ab; \label{H4}\\
		\pr(H_{a,b}=D_b)&=\frac{-a}{b-a}. \label{H5}
	\end{align}
\end{lemme}

	\section{Brownian bridges Coupling}\label{Section THm1}
	As announced we first deal with the proof of the two Theorems in the case where $x=x'$.
	 Let first consider some deterministic constants $T>0$, $\ph_0\in]0,i(M_k)[$ and $\theta_0\in]0,2\pi[$ with $i(M_k)$ the injective radius of $M_k$. We begin with the construction of a coupling strategy on $[0,T]$:
		\begin{prop}\label{couplage[0,T]}
	We fix $(N,e)\in TM_k$. We consider $X_t$ and $X_t'$ two Brownian motions on $M_k$ and $\left(\ph^{(N,e)}_t,\theta^{(N,e)}_t\right)$, $\left({\ph_t^{(N,e)}}',{\theta_t^{(N,e)}}'\right)$ their spherical/polar coordinates relative to $(N,e)$.  We suppose that $\ph^{(N,e)}_0={\ph_0^{(N,e)}}'=\ph_0$ and $\theta_0^{(N,e)}={\theta_0^{(N,e)}}'=\theta_0$. We also consider the swept area $I_t{(N,e)}$ and $I'_t{(N,e)}$, as defined in Lemma \ref{aireBalayee}, starting from $z_0$ and $z_0'$ respectively such that $z_0-z_0'\in ]-4\pi,4\pi[$.\\
	There exists a coupling of $X_t$ and $X_t'$ such that $X_T=X_T'$ a.s. and such that, for $T$ small enough, we have: \begin{equation*}\min_{z_0,z_0'}\left(\pr\big(I_T{(N,e)}-I_T'{(N,e)}\equiv 0\mod (4\pi)\big)\right)>0.\end{equation*}
	\end{prop}
	\subsection{Construction of the coupling on $[0,T]$}\label{Section Br}
	\begin{proof}[Proof of Proposition \ref{couplage[0,T]}]
	To simplify the notations and as the change of coordinates induced by $(N,e)$ doesn't intervene in this section, during this proof we will simply denote $\left(\ph^{(N,e)}_t,\theta^{(N,e)}_t,I_t{(N,e)}\right)$ by $\left(\ph_t,\theta_t,I_t\right)$ and $\left({\ph_t^{(N,e)}}',{\theta_t^{(N,e)}}',I'_t{(N,e)}\right)$ by $\left(\ph_t',\theta'_t,I'_t\right)$. By exchanging the roles of $X_t$ and $X_t'$ if needed, we can suppose that $z_0-z_0'>0$.\\
	
	We first chose $B_t^1=B_t^{1'}$ and thus $\ph_t=\ph_t'$. We define the change of time $\sigma(t)=\int_0^t\frac{k}{\sin^2(\sqrt{k}\ph_s)}ds$, there exists two Brownian motions $\Bet$ and $\Bet'$ adapted to the filtration $\left(\mathcal{F}_{\sigma(t)}\right)_t$ such that:
	\begin{equation*}\begin{cases}
		\theta_t=\theta_0+\Bet_{\sigma(t)}\\
		\theta'_t=\theta_0+\Bet'_{\sigma(t)}
	\end{cases}.\end{equation*}
	As in the coupling described in~\cite{banerjee2017coupling}, we are going to couple $\Bet$ and $\Bet'$ using Brownian bridges. Knowing all the path of $(\ph_t)_{t\in[0,T]}$ we define, for $\sigma\in [0,\sigma(T)]$:
	\begin{equation*}\begin{cases}
		\Bet_{\sigma}=\bbr_{\sigma}+ \frac{\sigma}{\sigma(T)}G \\
		\Bet'_{\sigma}=\bbrt_{\sigma}+ \frac{\sigma}{\sigma(T)}G \\
	\end{cases}
	\end{equation*}
	with $\bbr$ and $\bbrt$ two Brownian bridges on $[0,\sigma(T)]$ and $G$ a  Gaussian variable with mean $0$ and variance $\sigma(T)$, independent of the Brownian bridges. This way we will be able to define $B_t^{2}$ (resp. $B_t^{2'}$) such that $dB_t^{2}=\frac{\sin(\sqrt{k}\ph_t)}{\sqrt{k}}d\Bet_{\sigma(t)}$ (resp. $dB_t^{2'}=\frac{\sin(\sqrt{k}\ph_t)}{\sqrt{k}}d\Bet'_{\sigma(t)}$). 
	Using the Karhunen Loève decomposition of the Brownian bridges, for $\sigma\in[0,\sigma(T)]$, we can write: \begin{equation}\label{KL}\bbr_{\sigma}=\sqrt{\sigma(T)}\sum\limits_{j\geq1}Z_j\frac{\sqrt{2}}{j\pi}\sin\left(\frac{j\pi \sigma}{\sigma(T)}\right)
	\end{equation}  
	\begin{equation}
		\left(\text{resp. }\bbrt_{\sigma}=\sqrt{\sigma(T)}\sum\limits_{j\geq1}Z'_j\frac{\sqrt{2}}{j\pi}\sin\left(\frac{j\pi \sigma}{\sigma(T)}\right)\right)
	\end{equation} with $(Z_j)_j$ (resp. $(Z'_j)_j$) a sequence of independent standard Gaussian variables, independent of $B^1$. Note that, because of this independence with $B^1$, knowing $(B_s^1)_{s\in[0,T]}$, $B_t^{2}=\int_0^t \frac{\sqrt{k}}{\sin(\sqrt{k}\ph_s)}d\Bet_{\sigma(s)}$ defines an almost surely continuous process with independents increments and such that $B_t^{2}\overset{\mathcal{L}}{\sim} \mathcal{N}(0,t)$, that is a Brownian motion. As that distribution doesn't depend of the conditioning, $B^1$ and $B^2$ are two independent Brownian motions and our coupling is well defined.\\ 
	We now explain how we chose $(Z_j)_j$ and $(Z_j')_j$.
	If we take $Z_j=Z_j'$ for all $j\geq 2$, we get:
	\begin{equation*}\Bet_{\sigma}-\Bet'_{\sigma}=(Z_1-Z_1')\frac{\sqrt{2\sigma(T)}}{\pi}\sin\left(\frac{\pi \sigma}{\sigma(T)}\right).\end{equation*}
	Note that, with this choice of $(Z_j,Z_j')_{j\geq2}$, $X_t$ and $X_t'$ are equal only for $t\in\{0,T\}$.
	When we look at the impact of the choice of $(Z_1,Z_1')$ on the swept areas, we have: \begin{align*}
		I_t-I_t'&=z_0-z_0'+\int_0^t \frac{\tan\left(\frac{\sqrt{k}\ph_s}{2}\right)}{\sqrt{k}}\frac{\sin(\sqrt{k}\ph_s)}{\sqrt{k}}(Z_1-Z_1')\frac{\sqrt{2\sigma(T)}}{\pi}d\left(\sin\left(\frac{\pi \sigma(s)}{\sigma(T)}\right)\right) \\
		&=z_0-z_0'+\int_0^t \frac{1}{k}\left(1-\cos\left(\sqrt{k}\ph_s\right)\right)(Z_1-Z_1')\frac{\sqrt{2\sigma(T)}}{\pi}\cos\left(\frac{\pi \sigma(s)}{\sigma(T)}\right)\frac{\pi d\left(\sigma(s)\right)}{\sigma(T)}\\
		&=z_0-z_0'+\int_0^t \left(1-\cos(\sqrt{k}\ph_s)\right)(Z_1-Z_1')\sqrt{\frac{2}{\sigma(T)}}\cos\left(\frac{\pi \sigma(s)}{\sigma(T)}\right)\frac{ds}{\sin^2(\sqrt{k}\ph_s)}\\
		&=z_0-z_0'+K(t)\frac{Z_1-Z_1'}{2} \\
		&\text{ with }K(t)=2\sqrt{\frac{2}{\sigma(T)}}\int_0^t \frac{1}{1+\cos(\sqrt{k}\ph_s)}\cos\left(\frac{\pi \sigma(s)}{\sigma(T)}\right)ds.
	\end{align*}
	In order to obtain a successful coupling at time $T$, we need $I_T-I'_T\equiv 0 \mod (4\pi)$, that is $K(T)\frac{Z_1-Z_1'}{2}\equiv -(z_0-z_0') \mod (4\pi)$.
	Let's take $(W_t)_t$ a Brownian motion independent of $B^1$, $G$ and $(Z_j)_{j\geq 2}$. We define $\varsigma:=\inf\{t | W_t\notin ]-\frac{z_0-z_0'}{K(T)}, \frac{-(z_0-z_0')+4\pi}{K(T)}[\}$ and $W'_t:=
	\begin{cases}
		-W_t &\text{ if } t\leq \varsigma \\\
		W_t -2W_{\varsigma}  &\text{else}
	\end{cases}.$
	Note that, by the strong Markov property, $(W'_t)_t$ is a real Brownian motion starting at $0$ and independent of $\varsigma$. We then chose $Z_1=W_1\sim\mathcal{N}(0,1)$ and $Z_1'=W_1'\sim\mathcal{N}(0,1)$.
	
	In fact, with this construction we have: $\frac{Z_1-Z_1'}{2}=W_{1\wedge \varsigma}$. Thus, we get two cases:
	\begin{itemize}
		\item If $\varsigma\leq 1$, then $K(T)\frac{Z_1-Z_1'}{2}=
		K(T)W_{\varsigma}\equiv -(z_0-z_0')(4\pi)$.
		\item If $\varsigma> 1$, then $K(T)\frac{Z_1-Z_1'}{2}=K(T)W_1\not\equiv -(z_0-z_0')\mod (4\pi)$.
	\end{itemize}
	We have $\pr(I_T-I_T'\equiv 0\mod (4\pi))=\pr(\varsigma\leq 1)$. For this probability to be positive, we need to ensure that $K(T)=0$ does not occur a.s. This can be obtained from the following Lemma:
		\begin{lemme}\label{K(T)Tpetit}
		Let define $A(T):=\sqrt{\frac{2}{T}}\int_0^T\sin\left(\frac{\pi t}{T}\right)dB_t^1$. Then, for every curvature $k\in\mathbb{R}$, we have:
		\begin{equation*}
		K(T)=-\frac{2T}{\pi}A(T)+o\left(T^{\frac{3}{2}}\ln\left(\frac{1}{T}\right)\right)
		\end{equation*}
		with $o$ the Landau's notation for an a.s. convergence with $T$ close to $0$.
		In particular, as $A(T)$ has a standard Gaussian distribution, we have $\frac{\pi K(T)}{2T}\xrightarrow[T\to 0]{\mathcal{L}}\mathcal{N}(0,1)$ and, thus, $\pr(K(T)=0)\xrightarrow[T\to 0]{}0$.
	\end{lemme} 
    The proof  of Lemma \ref{K(T)Tpetit} will be given in subsection \ref{Lemmas}.
    From Lemma \ref{K(T)Tpetit}, we get that $-\frac{z_0-z_0'}{K(T)}$ and $\frac{-(z_0-z_0')+4\pi}{K(T)}$ are finite with a non zero probability for $T$ small enough. Moreover, by construction, $K(T)$ is independent of $(W_t)_t$. Thus we have:
		\begin{equation*}
		    0<\pr(\varsigma\leq 1).
		\end{equation*}
		
	Note that, with this strategy, as $X_t$ and  $X_t'$ only meet at time $0$ or $T$, the coupling is successful after time $T$ if and only if $\varsigma>1$. Note also that, by defining $\tilde{\varsigma} :=\inf\{t |\ |W_t|= \frac{4\pi}{|K(T)|}\}$, we have $\tilde{\varsigma}\geq\varsigma$ and thus $\pr(\varsigma\leq 1)$ is bounded below by $\pr(\tilde{\varsigma}\leq 1)$ that does not depend of the starting points $(z_0,z'_0)$.
	
		This end the proof of Proposition \ref{couplage[0,T]}.
	\end{proof}
	
	\subsection{Proof of Theorem \ref{Succes2}}\label{ThmFibres}
	We now have all the tools to construct the successful coupling if the starting points are in the same fiber.
	We first begin with the construction of an exponentially decreasing successful coupling without dependence with the starting points of the Brownian motion. We remind that $E_k$ denotes $SU(2)$ and $SL(2,\mathbb{R})$ depending of the value of $k$.
	\begin{prop}\label{Exponential}
	    	Let $g=(x,z)$, $g'=(x',z')\in E_k$. We suppose that $x=x'$.\\
		There exists a non co-adapted successful coupling of Brownian motions $(\bbb_t,\bbb_t')$ on $E_k$ starting at $(g,g')$ and $T$, $\tilde{C}$, $\tilde{c}$ some non negative constants that do not depend on the starting points of the processes, such that, for all $t>T$:
		\begin{equation}
			\pr(\tau>t)\leq \tilde{C}e^{-\tilde{c} t}.
		\end{equation}
	\end{prop}
	\begin{proof}[Proof of the Proposition \ref{Exponential}]
		To define the coupling on $[0,+\infty[$, we divide the time in intervals $[t_n,t_{n+1}[$ with length $T_n$ small enough, as in Lemma \ref{K(T)Tpetit}, and we repeat the coupling from Proposition \ref{couplage[0,T]}. As we proved that the probability of success at time $T_n$ is non zero, reproducing this strategy identically and independently on each interval $[t_n,t_{n+1}[$ should be efficient.\\
		With this in mind, we consider $(T_n)_n$ constant with $T_n=T$. We define $K^n(T)$, $W^n_t$ and $\varsigma^n$ the objects used in the construction of the coupling from Proposition \ref{couplage[0,T]} for each interval $[t_n,t_{n+1}[$.  It is true that the experiments will not be identical non independent as $\left(\ph^{(N_0,e_0)}_{t_{n+1}},\theta^{(N_0,e_0)}_{t_{n+1}},I_{t_{n+1}}{(N_0,e_0)}\right)$ is, in general, non constant and dependent of\\
		{$\left(\ph^{(N_0,e_0)}_{t_{n}},\theta^{(N_0,e_0)}_{t_{n}},I_{t_{n}}{(N_0,e_0)}\right)$}. To avoid this problem, the idea is to change the spherical/polar coordinate system on each interval of time $]t_n,t_{n+1}[$ by considering a sequence of tangent vectors ${(N_n,e_n)}_n$ such that the new sequence of coordinates $\left(\ph^{(N_n,e_n)}_{t_{n}},\theta^{(N_n,e_n)}_{t_{n}}\right)_n$ stays constant equal to $(\ph_0,\theta_0)$. 

		To obtain a successful coupling on $SU(2)$ (resp. $SL(2,\mathbb{R})$), we need to obtain $\zeta_{t_n}\equiv 0 \mod(4\pi)$ for some $n$. It is true that, for any $(N,e)\neq (N_0,e_0)$, we have in general  $\zeta_t\neq I_t{(N,e)}-I'_t{(N,e)}$. However, using Remark \ref{aireBalayee3}, we have $\zeta_{t}\equiv I'_{t}{(N_n,e_n)}-I_{t}{(N_n,e_n)} \mod(4\pi)$ at last at times $t=t_n$ for all $n$ (because $X_{t_n}=X'_{t_n}$ for all $n$). Thus, the coupling is successful if and only if there exists $n$ such that $I'_{t_n}{(N_n,e_n)}-I_{t_n}{(N_n,e_n)}\equiv0 \mod(4\pi)$.\\
We consider the variables $\tilde{\varsigma}^n :=\inf\{t |\ |W_t|\leq \frac{4\pi}{|K^n(T)|}\}$ as introduced in the last part of Proposition \ref{couplage[0,T]}. By choice of $(T_n)_n$ and $(N_n,e_n)_n$, we have $(K^n(T))_n$ independent and identically distributed and thus $(\tilde{\varsigma}^n)_n$ is too. In particular $\tilde{\varsigma}^n\geq \varsigma^n$ for all $n$. 
		This way  we get: 
		\begin{align}
			\pr(\tau>t_n)&=\pr(\varsigma^i>1 \ \forall\ 0\leq i\leq n-1)\notag\\
			&\leq\pr(\tilde{\varsigma}^i>1 \ \forall\ 0\leq i\leq n-1)=\pr(\tilde{\varsigma}^0>1)^n. \label{geometrique}
		\end{align}
		This last quantity tends to zero when $n\to +\infty$, thus $\tau$ is finite a.s., the coupling is successful and the coupling rate is clearly exponentially decreasing.
 		More precisely, we obtain for $t\in[t_n,t_{n+1}[$:\begin{align*}
 			\pr(\tau>t)&\leq \pr(\tau>t_n)\leq \exp\left(-n \ln\left(\frac{1}{\pr(\tilde{\varsigma}^0>1)}\right)\right)\\
			&=\frac{1}{\pr(\tilde{\varsigma}^0>1)}\exp\left(-(n+1)T \frac{\ln\left(\frac{1}{\pr(\tilde{\varsigma}^0>1)}\right)}{T}\right) \leq \tilde{C}\exp\left(-t\tilde{c}\right)\end{align*}
		with $\tilde{C}= \frac{1}{\pr(\tilde{\varsigma}^0>1)}$ and $\tilde{c}= \frac{1}{T}\ln\left(\frac{1}{\pr(\tilde{\varsigma}^0>1)}\right)$.\\
	Note that, if we change the system of coordinates $(N_0,e_0)$ on $S^2$ at the first step, we can chose $(\ph_{t_n}^{(N_n,e_n)},\theta_{t_n}^{(N_n,e_n)})_n$ constant equal to a value chosen independent of the initial position $(\ph_0,\theta_0)$ of the Brownian motions. Thus the random variables $K^n(T)$ and ${\tilde{\varsigma}}^n$ do not depend of these starting points and $\pr(\tilde{\varsigma}^0>1)$ neither. The coupling rate obtained in this case does not depend of the starting points.
	\end{proof}
	We now want to study how the coupling built above depends on the starting points of the Brownian motions.
	\begin{prop}\label{majoration}
	    Let $g=(x,z)$, $g'=(x',z')\in E_k$. We suppose that $x=x'$. We choose $t_f>0$.\\
		There exists a non co-adapted  coupling of Brownian motions $(\bbb_t,\bbb_t')$ on $E_k$ starting at $(g,g')$ such that 
		for all $0<q<1$, there exists some non negative constant $C$ that does not depend on the starting points $g$ and $g'$ satisfying for all $t>t_f$:
	 \begin{equation*}
	 \pr(\tau>t_f)\leq C|\zeta_0|\ln\left(\frac{1}{|\zeta_0|}\right) \text{ for }|\zeta_0|\text{ small enough}.\end{equation*}
		In particular, there exists some non negative constant $\tilde{C}_q$ that does not depend on the starting points $g$ and $g'$ satisfying for all $t>t_f$:
	 \begin{equation*}\label{majorationPuissance}
	 \pr(\tau>t_f)\leq \left(\tilde{C}_q\times|\zeta_0|^{q}\right)\wedge 1.\end{equation*}
	\end{prop}
	\begin{proof}[Proof of Proposition \ref{majoration}]
	Without loosing generality, we can still suppose that $z_0-z_0'\in]-4\pi,4\pi]$. Then, from Proposition \ref{equiv}, as $X_0=X_0'$, we have exactly $\zeta_0= z_0-z_0'$, thus we need to prove:
	\begin{equation}\label{majorationZ}
	\pr(\tau>t_f) \leq C|z_0-z_0'|\ln\left(\frac{1}{|z_0-z_0'|}\right).\end{equation}
	 Let $n$ be an integer that we will precise later, we consider $T_f=\frac{t_f}{n}$. We construct on $[0,t_f]$ the coupling described in Proposition \ref{Exponential} using the decomposition in the intervals $[t_j,t_{j+1}]$ with $t_j:=jT_f$ and $t_{j+1}:=(j+1)T_f$. 
			We simply denote by $I_t-I_t'$ the concatenation of all the $\left(I_t(N_{j+1})-I'_t(N_{j+1})\right)_{t\in[t_j,t_{j+1}[}$.
			Observing that, for $t\in[0,t_f]$, we have: \begin{equation}
				I_t-I_t'=z_0-z_0'+\sum_{j=0}^{n-1}K^j(t\wedge T_f)W^j_{1\wedge\varsigma^j},
			\end{equation}
			we define: \begin{equation}M_t:=z_0-z_0'+\sum_{j=0}^{n-1}K^j(T_f)W^j_{1\wedge(\frac{t-t_j}{T_f})}\mathbb{1}_{t\geq t_j} .
			\end{equation}
			On the event $\tau>t_f$, we have $\varsigma^j>1$ for all $j\leq n-1$,  and  $M_{t}=I_{t}-I_{t}'$ at times $t_j$. Moreover, by construction of $\varsigma_i$, $M_t\not\equiv 0\mod(4\pi)$ for all $t\in[0,t_f]$. As a consequence, $\tau>t_f$ if and only if $M_t\not\equiv 0\mod(4\pi)$ for all $t\leq t_f$.\\
			As $M_t$ is a martingale, for all $t\in[0,t_f]$, using the change of time defined by $S(t):=\sum\limits_{j=0}^{n-1}{K^j(T_f)}^2\frac{t-t_j}{T_f}\mathbb{1}_{t\geq t_j}$, we can write $M_t=z_0-z_0'+C_{S(t)}$ with $C$ a Brownian motion starting at $0$. As in Lemma \ref{hittingTime}, we denote $D_{-(z_0-z_0')}:=\inf\{t>0| C_t=-(z_0-z_0')\}$ and we get:
			\begin{align}\label{chgmt temps}
				\pr(\tau>t_f)&=\pr(z_0-z_0'+C_{S(s)}\in (0,4\pi) \text{ for all } s\leq t_f)\\
				&\leq \pr(D_{-(z_0-z_0')}>S(nT_f))
				=\pr\left(D_{-(z_0-z_0')}>\sum\limits_{j=0}^{n-1}{K^j(T_f)}^2\right).
			\end{align}
			 We separate the cases where $|K^j(T_f)|$ is large enough and the cases where it is not.
			From Lemma \ref{K(T)Tpetit}, using the convergence in law, there exists $0<\eps<1$ and $T_0>0$ such that for all $T\leq T_0$, $\pr\left(\frac{\pi|K(T)|}{2T}\leq\frac{1}{2}\right)<\eps$. Supposing that $n$ is large enough, we have $T_f\leq T_0$ and thus:
			\begin{align*}
			    \pr(\tau>t_f)&=\pr\left(\{\tau>t_f\}\cap\left\{ \exists j\in\{0,...,n-1\} \ \big| \  \frac{\pi|K^j(T_f)|}{2T_f}>\frac{1}{2}\right\}\right)\\
			    &+\pr\left(\left\{\tau>t_f\right\}\cap\left\{ \forall j\in\{0,...,n-1\} ,\  \frac{\pi|K^j(T_f)|}{2T_f}\leq\frac{1}{2}\right\}\right)\\
			    &\leq\pr\left(D_{-(z_0-z_0')}>\sum\limits_{j=0}^{n-1}{K^j(T_f)}^2>\frac{T_f^2}{\pi^2}\right)+\pr\left(\frac{\pi|K^j(T_f)|}{2T_f}\leq\frac{1}{2}\right)^n\\
			    &\leq \frac{\pi |z_0-z_0'|}{T_f}+\eps^n
			\end{align*}  
			where we use Lemma \ref{hittingTime} to get the left hand side term. Finally, we have \begin{equation*} \pr(\tau>t_f)\leq \frac{\pi |z_0-z_0'|}{t_f}n+\eps^n.\end{equation*}
			 If we chose $n$ such that $\frac{\ln(|z_0-z_0'|)}{\ln(\eps)}\leq n\leq \frac{\ln(|z_0-z_0'|)}{\ln(\eps)}+1$, we get:
			\begin{itemize}
			    \item $\eps^n< |z_0-z_0'|$;
			    \item $T_f=\frac{t_f}{n}\leq t_f \frac{\ln(\eps)}{\ln(|z_0-z_0'|)}$,and thus, $T_f\leq T_0$ for $z_0-z_0'$ small enough;
			    \item $\frac{\pi |z_0-z_0'|}{t_f}n\leq \frac{\pi |z_0-z_0'|}{t_f}\left(\frac{\ln(|z_0-z_0'|)}{\ln(\eps)}+1\right).$
			\end{itemize}
			We thus obtain the inequality (\ref{majorationZ}) for $|z_0-z_0'|$ small enough. Note that the obtained constant does depend of the chosen time $t_f$.
			\end{proof}
			\begin{NB}
			 Note here that, if $\frac{1}{\sqrt{\sum\limits_{j=0}^{n-1}{K^j(T_f)}^2}}$ is integrable, then, $\pr(\tau>nT_f)\leq C |z_0-z_0'|$ (with $C$ independent of the starting points) which would be better than the expected inequality. Here, contrary to the case of the Heisenberg group dealt in~\cite{banerjee2017coupling} we have not been able to prove this integrability. 
			\end{NB}
			\begin{NB}	The process $(M_t)_t$ introduced in the above proof is the one used in~\cite{banerjee2017coupling} to deal with the case of the Heisenberg group. We can also use it to obtain a proof of Proposition \ref{Exponential}. Denoting $H=\inf\{t>0| C_t\notin]-(z_0-z_0'),4\pi-(z_0-z_0')[\}$, we get:
		\begin{align*}\pr(\tau>t_n)&=\pr(z_0-z_0'+C_{S(s)}\in (0,4\pi) \text{ for all } s\leq t_n)\\
		&=\pr(H>S(t_n))=\pr\left(H>\sum\limits_{k=0}^{n-1} {K^k(T)}^2\right).
		\end{align*}
		
		Taking some $\delta>0$ such that $\sqrt{\frac{\delta}{2}}\times4\pi\not\equiv \frac{\pi}{2}\mod (\pi)$,
		and using Lemma \ref{tpsSortie}:
		\begin{equation*}
			\pr(H>u)=\esp[e^{\delta H}e^{-\delta H}\mathbb{1}_{H>u}]\leq e^{-\delta u}\esp[e^{\delta H}]\leq e^{-\delta u}\frac{\cos\left(\sqrt{\frac{\delta}{2}}(4\pi-2(z_0-z_0'))\right)}{\cos\left(\sqrt{\frac{\delta}{2}}\times 4\pi\right)}.
		\end{equation*}
		Then, $\pr(\tau>t)\leq \esp [e^{-\delta S(t_n)}]\frac{1}{\cos\left(\sqrt{\frac{\delta}{2}}\times 4\pi\right)}$.
		As $(K^k(T))_k$ is a sequence of independent and identically distributed variables, we get:
		\begin{equation*}\esp [e^{-\delta S(t_n)}]=\prod\limits_{k=0}^{n-1} \esp\left[e^{-\delta (K^k(T))^2}\right]=\esp\left[e^{-\delta ({K^0(T)})^2}\right]^n\leq e^{nT \frac{\ln\left(\esp\left[\exp\left(-\delta({K^0(T)})^2\right)\right]\right)}{T}}.\end{equation*}
		As $\pr\left({K^0(T)}^2=0\right)<1$, we have $\esp\left[\exp\left(-\delta{K^0(T)}^2\right)\right]<1$ and $\esp \left[e^{-\delta S(t_n)}\right]\leq e^{-nT \frac{c(\delta,T)}{T}}$ with $c(\delta,T)=-\ln\left(\esp\left[\exp\left(-\delta{K^0(T)}^2\right)\right]\right)>0$. This gives the expected rate of convergence.
	\end{NB}
			We can now give the final construction of the successful coupling from Theorem \ref{Succes2}:
			
			\begin{proof}[Proof of Theorem \ref{Succes2}]
			We first use the coupling from Proposition \ref{majoration} on $[0,t_f]$ and, then we construct the rest of the coupling using Proposition \ref{Exponential} on $[t_f,\tau]$. We have, for $t>t_f$:
			\begin{align*}
				\pr(\tau>t)&=\pr(\tau>t_f)\pr(\tau>t|\tau>t_f)\\
				&\leq \tilde{C}_q\times|\zeta_0|^{q} \tilde{C}\exp(-(t-t_f)\tilde{c}).
			\end{align*}
			As $X_0=X_0'$, we have $d_{cc}(\bbb_0,\bbb'_0)\sim \sqrt{|{\zeta_0}|}$. This give the expected inequality.
		\end{proof}

	\subsection{Proof of Lemma \ref{K(T)Tpetit}}\label{Lemmas}
	
	\begin{proof}
		The proof is using series expansion for $T$ close to $0$. In all that follow Landau's notations $o$ and $\mathcal{O}$ are used for an a.s. convergence with $T$ close to $0$. We give the proof for $k\neq 0$ but note that the same method can be used for $k=0$.
	Let $t\in[0,T]$. We first claim that: \begin{align}
				\sigma(t):&=\int_0^t\frac{k}{\sin^2(\sqrt{k}\ph_s)}ds\notag\\
				&=\frac{kt}{\sin^2(\sqrt{k}\ph_0)}\left(1-2\sqrt{k}\cot(\sqrt{k}\ph_0)\frac{1}{t}\int_0^tB^1_sds+o\left(T\ln\left(\frac{1}{T}\right)\right)\right).\label{DLsigma}
			\end{align}
			Indeed for $s\in[0,t]$, using Itô's relation, we have:
			\begin{align*}
				\sin(\sqrt{k}\ph_s)=\sin(\sqrt{k}\ph_0)&+\int_0^s\sqrt{k}\cos(\sqrt{k}\ph_u)dB^1_u\\
				&+\frac{k}{2}\int_0^s\left(-\sin(\sqrt{k}\ph_u)+\frac{\cos^2(\sqrt{k}\ph_u)}{\sin(\sqrt{k}\ph_u)}\right)du.
			\end{align*}
			We remind that, using the law of the iterated logarithm, we have, for $s$ small enough: $B_s=o\left(\sqrt{s\ln\left(\frac{1}{s}\right)}\right)$. More generally, if we consider the martingale $M_s=\int_0^sv(\omega,u)dB^1_u$, for $s\to 0$, we have:
			\begin{equation}\label{lli}
				M_s=o\left(\sqrt{\langle M_s,M_s\rangle\ln\left(\frac{1}{\langle M_s,M_s\rangle}\right)}\right).
			\end{equation}
			Indeed, we just have to use Dambis-Dubins-Schwartz theorem to write $M_s$ as a time-changed Brownian motion. Then the law of iterated logarithm gives the attended result. Thus, for $s\to 0$, using (\ref{lli}) for $v(u)=\sqrt{k}\cos(\sqrt{k}\ph_u)$ and remarking that $\int_0^s v(u)^2du=\mathcal{O}(s)$ (we use the continuity of $\ph$ and the compacity of $[0,T]$), we get:
			 \begin{equation*}\int_0^s\sqrt{k}\cos(\sqrt{k}\ph_u)dB^1_u=o\left(\sqrt{s\ln\left(\frac{1}{s}\right)}\right).\end{equation*}
			Thus:
			\begin{equation*}
				\sin^2(\sqrt{k}\ph_s)=\sin^2(\sqrt{k}\ph_0)+2\sin(\sqrt{k}\ph_0)\int_0^s\sqrt{k}\cos(\sqrt{k}\ph_u)dB_u^1+o\left(s\ln\left(\frac{1}{s}\right)\right)
			\end{equation*}
			The same way, using Itô's formula and relation (\ref{lli}), we have:
			\begin{align}\label{cos} 
				\cos(\sqrt{k}\ph_u)&=\cos(\sqrt{k}\ph_0)-\int_0^u \sqrt{k}\sin(\sqrt{k}\ph_v)dB_v^1-k\int_0^u\cos(\sqrt{k}\ph_v)dv\\
				&=\cos(\sqrt{k}\ph_0)+\eps(u).\notag
			\end{align}
			with $\eps(u)=o\left(\sqrt{u\ln\left(\frac{1}{u}\right)}\right)$. In particular, $\int_0^s\eps(u)^2du=o\left(s^2\ln\left(\frac{1}{s}\right)\right)$. Thus, applying (\ref{lli}) to $\int_0^s\eps(u)dB_u^1$, we get: \begin{equation*}
				\int_0^s\sqrt{k}\cos(\sqrt{k}\ph_u)dB_u^1
				=\sqrt{k}\cos(\sqrt{k}\ph_0)B_s^1+o\left(s\ln\left(\frac{1}{s}\right)\right).
			\end{equation*}
			Finally we obtain:
			\begin{align*}
				\sin^2(\sqrt{k}\ph_s)&=\sin^2(\sqrt{k}\ph_0)+2\sqrt{k}\sin(\sqrt{k}\ph_0)\cos(\sqrt{k}\ph_0)B_s^1+o\left(s\ln\left(\frac{1}{s}\right)\right)\\
				&=\sin^2(\sqrt{k}\ph_0)\left(1+2\sqrt{k}\cot(\sqrt{k}\ph_0)B_s^1+o\left(T\ln\left(\frac{1}{T}\right)\right)\right)\text{ as }s\leq T.
			\end{align*}
			and:
			\begin{equation*}
				\frac{1}{\sin^2(\sqrt{k}\ph_s)}
				=\frac{1}{\sin^2(\sqrt{k}\ph_0)}\left(1-2\sqrt{k}\cot(\sqrt{k}\ph_0)B_s^1+o\left(T\ln\left(\frac{1}{T}\right)\right)\right).
			\end{equation*}
			We just have to integrate this expression to obtain (\ref{DLsigma}).
		We then can deduce an estimation for $\cos\left(\frac{\pi\sigma(t)}{\sigma(T)}\right)$. Indeed, we have:
			\begin{equation}\label{sigma}
				\frac{1}{\sigma(T)}=\frac{\sin^2(\sqrt{k}\ph_0)}{kT}\left(1+2\sqrt{k}\cot(\sqrt{k}\ph_0)\frac{1}{T}\int_0^TB_s^1ds+o\left(T\ln\left(\frac{1}{T}\right)\right)\right).
			\end{equation}
			Thus, as $\frac{1}{t}\int_0^t B_s^1ds=o\left(\sqrt{T\ln\left(\frac{1}{T}\right)}\right)$ for all $0\leq t\leq T$, we obtain:
			\begin{equation*}
				\frac{\pi\sigma(t)}{\sigma(T)}=\frac{\pi t}{T}\left(1-2\sqrt{k}\cot(\sqrt{k}\ph_0)\left(\frac{1}{t}\int_0^t B_s^1ds-\frac{1}{T}\int_0^T B_s^1ds\right)+o\left(T\ln\left(\frac{1}{T}\right)\right)\right).
			\end{equation*}
			Finally: 
			\begin{align*}
				\cos\left(\frac{\pi\sigma(t)}{\sigma(T)}\right)&=\cos\left(\frac{\pi t}{T}\right)+\sin\left(\frac{\pi t}{T}\right)\frac{\pi t}{T}\times 2\sqrt{k}\cot(\sqrt{k}\ph_0)\left(\frac{1}{t}\int_0^t B^1_sds-\frac{1}{T}\int_0^T B^1_sds\right)\\
				&+o\left(T\ln\left(\frac{1}{T}\right)\right).
			\end{align*}
		Using same methods as in the first part of this proof, we have:
			\begin{align*}
				\sqrt{k}\sin(\sqrt{k}\ph_u)&=\sqrt{k}\sin(\sqrt{k}\ph_0)+\eps(u)\text{ with }\eps(u)=o\left(\sqrt{u\ln\left(\frac{1}{u}\right)}\right)\\
				\text{and so: }	\cos(\sqrt{k}\ph_t)&=\cos(\sqrt{k}\ph_0)-\int_0^t \sqrt{k}\sin(\sqrt{k}\ph_u)dB_u^1-k\int_0^u\cos(\sqrt{k}\ph_u)du\\
				&=\cos(\sqrt{k}\ph_0)-\sqrt{k}\sin(\sqrt{k}\ph_0)B_t^1+o\left(T\ln\left(\frac{1}{T}\right)\right).
			\end{align*}
			Then we get:
			\begin{align*}
				\frac{1}{1+\cos(\sqrt{k}\ph_t)}&=\frac{1-\cos(\sqrt{k}\ph_t)}{\sin^2(\ph_t)}\\
				&=\left(1-\cos(\sqrt{k}\ph_0)+\sqrt{k}\sin(\sqrt{k}\ph_0)B_t^1+o\left(T\ln\left(\frac{1}{T}\right)\right)\right)\\
				&\times\frac{1}{\sin^2(\sqrt{k}\ph_0)}\left(1-2\sqrt{k}\cot(\sqrt{k}\ph_0)B_t^1+o\left(T\ln\left(\frac{1}{T}\right)\right)\right)\\
				&=\frac{1-\cos(\sqrt{k}\ph_0)+\sqrt{k}\left(\sin(\sqrt{k}\ph_0)-2\cot(\sqrt{k}\ph_0)(1-\cos(\sqrt{k}\ph_0))\right)B_t^1}{\sin^2(\sqrt{k}\ph_0)}\\
				&+o\left(T\ln\left(\frac{1}{T}\right)\right).
			\end{align*}
		We can now finalize the calculation of $K(T)=\sqrt{\frac{2}{\sigma(T)}}\int_0^T \frac{2}{1+\cos(\sqrt{k}\ph_t)}\cos\left(\frac{\pi\sigma(t)}{\sigma(T)}\right)dt$.
			Using the previous results, we get:
			\begin{equation*}
				\frac{2}{1+\cos(\sqrt{k}\ph_t)}\cos\left(\frac{\pi\sigma(t)}{\sigma(T)}\right)=\frac{2}{\sin^2(\sqrt{k}\ph_0)}\big(R_1(t)+R_2(t)+R_3(t)\big)+o\left(T\ln\left(\frac{1}{T}\right)\right).
			\end{equation*}
			With 
			\begin{align*}
				R_1(t)&:=\left(1-\cos(\sqrt{k}\ph_0)\right)\cos\left(\frac{\pi t}{T}\right)\\
				R_2(t)&:=\sqrt{k}\left(\sin(\sqrt{k}\ph_0)-2\cot(\sqrt{k}\ph_0)\left(1-\cos(\sqrt{k}\ph_0)\right)\right)B_t^1\cos\left(\frac{\pi t}{T}\right)\\
				R_3(t)&:=\sqrt{k}\left(1-\cos(\sqrt{k}\ph_0)\right)\times 2\sin\left(\frac{\pi t}{T}\right)\frac{\pi t}{T}\cot(\sqrt{k}\ph_0)\left(\frac{1}{t}\int_0^tB_s^1ds-\frac{1}{T}\int_0^TB_s^1ds\right).
			\end{align*}
			In particular $\int_0^T R_1(t)dt$ vanishes. We also have:
			\begin{align*}
				\int_0^T&\sin\left(\frac{\pi t}{T}\right)\frac{\pi t}{T}\left(\frac{1}{t}\int_0^tB_s^1ds-\frac{1}{T}\int_0^TB_s^1ds\right)dt\\
				&=\int_0^T\frac{\pi }{T}\sin\left(\frac{\pi t}{T}\right)\int_0^t B_s^1dsdt-\frac{1}{T}\int_0^TB_s^1ds\left(\left[-t\cos\left(\frac{\pi t}{T}\right)\right]_0^T+\int_0^T\cos\left(\frac{\pi t}{T}\right)dt\right)\\
				&=\left(\left[-\cos\left(\frac{\pi t}{T}\right)\int_0^t B_s^1ds\right]_0^T+\int_0^T\cos\left(\frac{\pi t}{T}\right)B_t^1dt\right)-\int_0^TB_s^1ds\\
				&=\int_0^T\cos\left(\frac{\pi t}{T}\right)B_t^1dt.
			\end{align*}
			Thus $\int_0^T\left(R_2(t)+R_3(t)\right)dt=\sqrt{k}\sin(\sqrt{k}\ph_0)\int_0^TB_t^1 \cos\left(\frac{\pi t}{T}\right)dt$.
			As $\int_0^TB_t^1\cos\left(\frac{\pi t}{T}\right)dt=-\frac{T}{\pi}\int_0^T\sin\left(\frac{\pi t}{T}\right)dB_t^1$, we obtain:
			\begin{equation*}	\int_0^T\frac{2\cos\left(\frac{\pi\sigma(t)}{\sigma(T)}\right)}{1+\cos(\sqrt{k}\ph_t)}dt=-\frac{2\sqrt{k}T}{\sin(\sqrt{k}\ph_0)\pi}\int_0^T\sin\left(\frac{\pi t}{T}\right)dB_t^1+o\left(T^2\ln\left(\frac{1}{T}\right)\right).
			\end{equation*}
			Using (\ref{sigma}) we also have $\sqrt{\frac{1}{\sigma(T)}}=\frac{\sin(\sqrt{k}\ph_0)}{\sqrt{kT}}\left(1+o\left(\sqrt{T\ln\left(\frac{1}{T}\right)}\right)\right)$ and:
			\begin{equation*}
				K(T)=\frac{-2\sqrt{2T}}{\pi}\int_0^T\sin\left(\frac{\pi t}{T}\right)dB_t^1+o\left(T^{\frac{3}{2}}\ln\left(\frac{1}{T}\right)\right).
			\end{equation*}
			As $\int_0^T\sin^2\left(\frac{\pi t}{T}\right)dt=\frac{T}{2}$, the distribution of $A(T):=\sqrt{\frac{2}{T}}\int_0^T\sin\left(\frac{\pi t}{T}\right)dB_t^1$ is a standard Gaussian and $K(T)=-\frac{2T}{\pi}A(T)+o\left(T^{\frac{3}{2}}\ln\left(\frac{1}{T}\right)\right)$.
	\end{proof}
	\section{Successful coupling in $SU(2)$}\label{Section THm3}
	\subsection{Reflection coupling}\label{SubSecRef}
	
	We first explicit one possible construction for a successful coupling in $S^2$. As explained in \cite{ReflKuwada}, the reflection coupling is maximal in $S^2$. It can be constructed by different ways. See for example \cite{ReflKuwada,CranstonRefl} for a construction using projections, \cite{KendallSU(2)} for a construction using covariant derivatives. We can also obtain this coupling directly with the spherical coordinates $(\ph_t,\theta_t)$ and $(\ph'_t,\theta'_t)$:
	\begin{equation*}
		\begin{cases}
			d\ph_t= dB_t^1+\frac{1}{2}\cot(\ph_t)dt\\
			d\theta_t=\frac{1}{\sin(\ph_t)}dB_t^2
		\end{cases}
		\text{ and }	\begin{cases}
			d\ph_t'= d{B_t^1}'+\frac{1}{2}\cot(\ph_t')dt\\
			d\theta_t'=\frac{1}{\sin(\ph_t')}d{B_t^2}'
		\end{cases}.
	\end{equation*}
	Changing the pole and vector of reference in $S^2$ if needed,we can suppose that $\ph_0=\pi-\ph_0$, $\ph_0\in]0,\frac{\pi}{2}[$ and $\theta_0=\theta_0'$. To construct the reflection coupling, we take $B_t^1=-{B_t^1}'$ and $B_t^2={B_t^2}'$, we get $\ph_t=\pi-\ph_t'$ and $\theta_t=\theta_t'$ for all $t$. With this construction the paths of the two Brownian motions in $S^2$ have a symmetry with respect to the equator. In particular for $t\in[0,\tau_1]$, defining $\rho_t:=\rho\left((\ph_t,\theta_t),(\ph'_t,\theta'_t)\right)$, we have $\rho_t=\pi-2\ph_t$. We obtain:
\begin{equation}\label{DistanceRef}
    d\left(\frac{\rho_t}{2}\right)=-dB_t^1-\frac{1}{2}\tan\left(\frac{\rho_t}{2}\right)dt.
\end{equation}
This is the equation obtained in~\cite{KendallSU(2)} using the covariant derivatives to define the reflection coupling.
	\begin{prop}\label{reflexion}
	We consider the reflection coupling described above. For any $0\leq u<\frac{\pi}{2}$, let denote $T_u:=\inf\{t\ | \ \frac{\rho_t}{2}=u\}$. If $\frac{\rho_0}{2}< u<\frac{\pi}{2}$, then:
	\begin{equation}\label{espRef}
	\esp[T_0\wedge T_u]\leq \frac{\rho_0}{2}\left(2u-\frac{\rho_0}{2}\right).
	\end{equation}
	Moreover the reflection coupling is successful. If we denote by $\tau_1$ its first coupling time, we have:
	\begin{itemize}
	    \item $\esp[\tau_1]=\esp[T_0]\leq\frac{\rho_0}{2}\left(\pi-\frac{\rho_0}{2}\right)$.
	    \item there exists some constants $C, c>0$ independent of the distance between the starting points such that:
		\begin{equation*}
			\pr(\tau_1>t)\leq C\rho_0\frac{e^{-ct}}{t}.
		\end{equation*}
	\end{itemize}
	\end{prop}
\begin{proof}
We first prove inequality (\ref{espRef}).
Let $\frac{\pi}{2}> u>\frac{\rho_0}{2}$. Let define three processes:
\begin{itemize}
\item $W_t:=\frac{\rho_0}{2}+B_t$, with $B_t$ a standard Brownian motion;
\item $U_t:=u-|u-W_t|=\frac{\rho_0}{2}+\beta_t-L_t^{u}(W)$
with $\beta_t:=-\int_0^t \sgn(u-W_s)dB_s$ defining a standard Brownian motion and $L_t^{u}(W)$ the local time of $(W_t)_t$ in $u$.
\item $(V_t)_t$ starting at $\frac{\rho_0}{2}$ satisfying:
$dV_t=d\beta_t-\frac{1}{2}\tan(V_t)dt$.
In particular $(\rho_t)_t$ and $(V_t)_t$ have the same distribution.
\end{itemize}

We also define the stopping time $\tilde{T}_v:=\inf\{ t>0\ | \ V_t=v\}$ for $v\in\mathbb{R}$. We claim that $V_t\leq U_t$ for all $0\leq t\leq \tilde{T}_0\wedge \tilde{T}_u$. Let explain this fact.\\
We consider $\underline{T}:=\inf\{t>0\ | \ U_t=V_t\}$. As $\frac{\rho_0}{2}<u$, for $t$ small enough, we have $U_t<u$. Thus, $d(U_t-V_t)=\frac{1}{2}\tan(V_t)dt>0$ and $\underline{T}>0$ a.s. We suppose that $\underline{T}<\tilde{T}_0\wedge \tilde{T}_u$. Then, we have $U_{\underline{T}}=V_{\underline{T}}<u$ and, by continuity of $U$, there exists $\eps>0$ such that, for all $t\in[\underline{T}-\eps,\underline{T}]$, $U_t<u$. In particular, $L_{\underline{T}-\eps}^u(W)=L_{\underline{T}}^u(W)$. We obtain:
\begin{align*}
    0&=U_{\underline{T}}-V_{\underline{T}}\\
    &=U_{\underline{T}-\eps}-V_{\underline{T}-\eps}+\frac{1}{2}\int_{\underline{T}-\eps}^{\underline{T}}\tan(V_s)ds
\end{align*}
with $U_{\underline{T}-\eps}-V_{\underline{T}-\eps}>0$ and with $\frac{1}{2}\int_{\underline{T}-\eps}^{\underline{T}}\tan(V_s)ds>0$ since $\underline{T}<\tilde{T}_0$. We get a contradiction, thus $\underline{T}\geq \tilde{T}_0\wedge  \tilde{T}_u$ so $V_t<U_t$ on $]0,\tilde{T}_0\wedge \tilde{T}_u[$.\\
As $U_t=0$ if and only if $W_t\in\{0,2u\}$, we get:
\begin{equation*}
    \tilde{T}_0\wedge\tilde{T}_u\leq \inf\{t>0\ | \ W_t\notin]0,2u[\}:=H_{-\frac{\rho_0}{2},2u-\frac{\rho_0}{2}}.
\end{equation*}
Finally, as and $T_0\wedge T_u$ has the same distribution as $ \tilde{T}_0\wedge\tilde{T}_u$ we get:
\begin{equation*}
    \esp[T_0\wedge T_u]\leq \esp[H_{-\frac{\rho_0}{2},2u-\frac{\rho_0}{2}}]=\frac{\rho_0}{2}\left(2u-\frac{\rho_0}{2}\right).
\end{equation*}
We can now interest ourselves in the coupling rate of the reflection coupling. As, $\frac{\rho_t}{2}<\frac{\pi}{2}$ a.s., we have $\tau_1=T_0\wedge T_{\frac{\pi}{2}}$. Using the previous results, we have $\underline{T}_0\wedge \underline{T}_{\frac{\pi}{2}}\leq H_{a,b}$, with $a=-\frac{\rho_0}{2}$ and $b=\pi-\frac{\rho_0}{2}$. From Lemma \ref{tpsSortie}, we obtain:
	\begin{align*}
		\pr(\tau_1>t)&\leq\esp\left[\frac{e^{-\delta H_{a,b}}}{H_{a,b}}H_{a,b}e^{\delta H_{a,b}}\mathbb{1}_{H_{a,b}>t}\right]\\
		&\leq \frac{e^{-\delta t}}{t}\times\frac{(\pi-\frac{\rho_0}{2})\frac{\rho_0}{2}}{\cos^2\left(\sqrt{\frac{\delta}{2}}\pi\right)}
	\end{align*}
with $\delta$ such that $0<\sqrt{\frac{\delta}{2}}<\frac{1}{2}$ that is $0<\delta<\frac{1}{2}$. Taking $C=\frac{\pi}{2\cos^2\left(\sqrt{\frac{\delta}{2}}\pi\right)}$ and $c=\delta$, we have:
\begin{equation*}
	\pr(\tau_1>t)\leq\pr(S_1>t)\leq \rho_0C  \frac{e^{-c t}}{t}.
\end{equation*}
with $c$ and $C$ not depending of the initial distance between the Brownian motions. In particular, the coupling is successful.
	\end{proof}

\subsection{Proof of Theorem \ref{Succes1}}
We now deal with the coupling $\left(\bbb_t=(X_t,z_t),\bbb_t'=(X_t,z_t)\right)$ on $SU(2)$ announced in Theorem \ref{Succes1}. As explained before, the idea is to use reflection coupling until the first time ${\tau_1}$ such that $X_{\tau_1}=X_{\tau_1}'$ and then to use the Brownian bridges coupling to couple the "area parts". According to Theorem \ref{Succes2}, we need to have an estimation of the quantity $\esp[|\zeta_{\tau_1}|^{q}\wedge 1]$ for at least one $q\in]0,1[$. Thus we need the following Proposition:
\begin{prop}\label{AireBalayéeGnl}
	At the end of the reflection coupling we have, for $0<p<\frac{1}{2}$:
	\begin{align*}
		\esp[|\zeta_{\tau_1}|^{\frac{1}{2}+p}]&\leq\tilde{C}_pd_{cc}(\bbb_0,\bbb_0').
	\end{align*}
	with $\tilde{C}_p$ some constant independent of $\bbb_0$ and $\bbb_0'$.
\end{prop}

 \begin{proof}[Proof of Proposition \ref{AireBalayéeGnl}]
By construction of the reflection coupling using the change of pole $(N_1,e_1)$, the quantity $\mathcal{A}_{X_t,X_t',N_1}=0$ for all $t$. Then by Lemma \ref{aireBalayee}, we have $\zeta_t\equiv I_t'(N_1,e_1)-I_t(N_1,e_1)\ \mod(4\pi)$. As $\frac{\rho_t}{2}=\frac{\pi}{2}-\ph_t$, using the equation of $I_t'(N_1,e_1)-I_t(N_1,e_1)$ we get $\zeta_t\equiv\zeta_0-2\int_0^t \tan\left(\frac{\rho_s}{2}\right)dB_s^2\ \mod(4\pi)$ (this expression of $\zeta_t$ is the one obtained in \cite{KendallSU(2)}). We define $\phi_p:x\in]0,+\infty[\mapsto x^{\frac{1}{2}+p}$ with $0<p<\frac{1}{2}$. We are interested in the quantity $\esp\left[\phi_p\left(|\zeta_0-2\int_0^{\tau_1} \tan\left(\frac{\rho_s}{2}\right)dB_s^2|\wedge 4\pi\right)\right]$. As $\phi_p(x+y)\leq \phi_p(x)+\phi_p(y)$ for all $x,y>0$, \begin{align*}
     \esp\left[\phi_p\left(\left|\zeta_0-2\int_0^{\tau_1} \tan\left(\frac{\rho_s}{2}\right)dB_s^2\right|\wedge 4\pi\right)\right]&\leq \phi_p\left(\left|\zeta_0\right|\right)\\
     &+\esp\left[\phi_p\left(\left|2\int_0^{\tau_1} \tan\left(\frac{\rho_s}{2}\right)dB_s^2\right|\wedge 4\pi\right)\right].
 \end{align*}
 Since $|\zeta_0|$ is bounded by $2\pi$, we have $\phi_p(|\zeta_0|)\leq \delta_p \sqrt{|\zeta_0|}\leq \tilde{\delta}_p d_{cc}(\bbb_0,\bbb_0')$ with $\delta_p$ and $\tilde{\delta}_p$ two constants independent of $\bbb_0$ and $\bbb'_0$ .
 We now just need to prove that the quantity $\esp\left[\phi_p\left(\left|2\int_0^{\tau_1} \tan\left(\frac{\rho_s}{2}\right)dB_s^2\right|\wedge 4\pi\right)\right]$ is upper bounded by $\rho_0$ up to a multiplicative constant. It is obvious for $\rho_0$ large enough, thus we consider $\rho_0\leq m$, with $m>0$ chosen later in the proof. We can write:
 \begin{align*}
     \esp\bigg[\phi_p\bigg(\bigg|2\int_0^{\tau_1} \tan\Big(\frac{\rho_s}{2}\Big)dB_s^2&\bigg|\wedge 4\pi\bigg)\bigg]\\
     &=\int_0^{\phi_p\left(4\pi\right)}\pr\left(\phi_p\left(\left|2\int_0^{\tau_1} \tan\left(\frac{\rho_s}{2}\right)dB_s^2\right|\wedge 4\pi\right)>y\right)dy\\
     &=\int_0^{4\pi}\pr\left(\phi_p\left(\left|2\int_0^{\tau_1} \tan\left(\frac{\rho_s}{2}\right)dB_s^2\right|\right)>\phi_p(x)\right)\phi_p'(x)dx\\
     &=\int_0^{4\pi}\pr\left(\left|2\int_0^{\tau_1} \tan\left(\frac{\rho_s}{2}\right)dB_s^2\right|>x\right)\left(\frac{1}{2}+p\right)x^{p-\frac{1}{2}}dx.
 \end{align*}
 We set $0<\alpha<1$. We are going to split the integral into two parts:
 \begin{align}
    &\int_0^{\left(\frac{\rho_0}{2}\right)^{\frac{1}{\alpha}}}\pr\left(\left|2\int_0^{\tau_1} \tan\left(\frac{\rho_s}{2}\right)dB_s^2\right|>x\right)\phi_p'(x)dx\label{I1}\\
    \text{and}\notag\\
    &\int_{\left(\frac{\rho_0}{2}\right)^{\frac{1}{\alpha}}}^{4\pi}\pr\left(\left|2\int_0^{\tau_1} \tan\left(\frac{\rho_s}{2}\right)dB_s^2\right|>x\right)\phi_p'(x)dx.\label{I2}
 \end{align}
 By simply upper-bounding the probability by $1$, the quantity (\ref{I1}) can be upper-bounded by $\phi_p\left(\left(\frac{\rho_0}{2}\right)^{\frac{1}{\alpha}}\right)\leq\left(\frac{\rho_0}{2}\right)^{\frac{1+2p}{2\alpha}}$. For $\alpha\leq\frac{1}{2}+p$ and $m$ small enough, we have $\left(\frac{\rho_0}{2}\right)^{\frac{1+2p}{2\alpha}}\leq\frac{\rho_0}{2}$.
 To deal with the second quantity we look at the quantity $\pr\left(\left|2\int_0^{\tau_1} \tan\left(\frac{\rho_s}{2}\right)dB_s^2\right|>x\right)$. Using the same notations as in proposition \ref{reflexion}, we consider $\frac{\rho_0}{2}<u<\frac{\pi}{2}$.
 We have:
 \begin{align*}
     \pr\left(\left|2\int_0^{\tau_1} \tan\left(\frac{\rho_s}{2}\right)dB_s^2\right|>x\right)&=\pr\left(\left|2\int_0^{\tau_1} \tan\left(\frac{\rho_s}{2}\right)dB_s^2\right|>x,T_u<T_0\right)\\
     &+\pr\left(\left|2\int_0^{\tau_1} \tan\left(\frac{\rho_s}{2}\right)dB_s^2\right|>x,T_u> T_0\right)\\
 &\leq\pr(T_u<T_0)+\pr\left(\left|2\int_0^{T_u\wedge T_0} \tan\left(\frac{\rho_s}{2}\right)dB_s^2\right|>x\right)\\
 &\leq\pr(T_u<T_0)+\frac{1}{x^2}\esp\left[4\int_0^{T_u\wedge T_0} \tan^2\left(\frac{\rho_s}{2}\right)ds\right]\\
 &\leq\pr(T_u<T_0)+\frac{4\tan^2\left(u\right)}{x^2}\esp\left[T_u\wedge T_0\right].
 \end{align*}
 As the drift part in the equation of $\rho_t$ (given by (\ref{DistanceRef})) is negative, we have $T_u\geq \inf\{ t>0 \ | \ \frac{\rho_0}{2}-B_t^1=u\}$ and $T_0\leq \inf\{ t>0 \ | \ \frac{\rho_0}{2}-B_t^1=0\}$. Thus, using relation (\ref{H5}) from Lemma \ref{tpsSortie} with $a=-\frac{\rho_0}{2}$ and $b=u-\frac{\rho_0}{2}$, we get $\pr(T_u<T_0)\leq \frac{\rho_0}{2u}$.
From Proposition \ref{AireBalayéeGnl}, we have  $\esp\left[T_u\wedge T_0\right]\leq \frac{\rho_0}{2}\left(2u-\frac{\rho_0}{2}\right)$.
Thus:
 \begin{equation}\label{InegalitéAire2}
     \pr\left(\left|2\int_0^{\tau_1} \tan\left(\frac{\rho_s}{2}\right)dB_s^2\right|>x\right)\leq \frac{\rho_0}{2}\left(\frac{1}{u}+4\frac{\tan^2(u)}{x^2}\left(2u-\frac{\rho_0}{2}\right)\right).
 \end{equation}
  We introduce $\beta$ such that $0<\beta<1$. We first chose $u(x):=x^{\beta}$. If $\left(\frac{\rho_0}{2}\right)^{\frac{1}{\alpha}}< x< 1$ and $\beta<\alpha$, we have $\frac{\rho_0}{2}< u(x)<\frac{\pi}{2}$. Then we use inequality (\ref{InegalitéAire2}) with $u=u(x)$. We get:

 \begin{align*}
     \int_{\left(\frac{\rho_0}{2}\right)^{\frac{1}{\alpha}}}^{1}\pr\bigg(\left|2\int_0^{\tau_1} \tan\left(\frac{\rho_s}{2}\right)dB_s^2\right|&>x\bigg)\phi_p'(x)dx\\
     &\leq \frac{\rho_0}{2}\int_{\left(\frac{\rho_0}{2}\right)^{\frac{1}{\alpha}}}^{1}\left(x^{-\beta}+8\frac{\tan^2(x^{\beta})}{x^2}x^{\beta}\right)dx\\
     &\leq \delta_{\beta} \frac{\rho_0}{2}\int_0^{1}\left(x^{-\beta}+x^{3\beta-2}\right)x^{p-\frac{1}{2}}dx,
 \end{align*}
 with $\delta_{\beta}$ independent of $\rho_0$.
 The last integral is finite if and only if $p+\frac{1}{2}>\beta>\frac{1}{2}-\frac{p}{3}$. In particular, this is the case for $\beta=\frac{1}{2}$.
 If $x>1$, we use inequality (\ref{InegalitéAire2}) with $u\equiv 1$. We get:
  \begin{equation*}
     \int_{1}^{4\pi}\pr\bigg(\Big|2\int_0^{\tau_1} \tan\left(\frac{\rho_s}{2}\right)dB_s^2\Big|>x\bigg)\phi_p'(x)dx\leq \frac{\rho_0}{2}\int_{1}^{4\pi}\left(1+8\frac{\tan^2(1)}{x^2}\right)\phi_p'(x)dx.
 \end{equation*}
 We thus have (\ref{I2}) upper bounded by $\rho_0$ up to a multiplicative constant only depending of $m$, $\beta$ and $p$.
 \end{proof}

We can now give the proof of Theorem \ref{Succes1}:
\begin{proof}[Proof of Theorem \ref{Succes1}]
To construct our successful coupling on $SU(2)$, we first construct the reflection coupling described above until time $\tau_1$, that is until $X_{\tau_1}=X'_{\tau_1}$. Then, we use the coupling from Theorem \ref{Succes2} to deal with the swept area coordinates. As the coupling rate from the two couplings have an exponential decay, we obtain a successful coupling with an exponential decay on $SU(2)$. We want to obtain the initial distance-control from inequality (\ref{inegalitédcc1}). We denote $\tau_2:=\inf\{t>\tau_1\ | \ \bbb_t=\bbb_t'\}$. The first meeting time $\tau$ of the Brownian motions in $SU(2)$ satisfies $\tau=\tau_1+\tau_2$. Then we have, for $\frac{1}{2}<q<1$:
\begin{align*}
	\pr(\tau>t)&=\pr\left(\tau>t,\tau_1\leq\frac{t}{2}\right)+\pr\left(\tau>t,\tau_1>\frac{t}{2}\right)\leq\pr\left(\tau_2>\frac{t}{2},\tau_1\leq\frac{t}{2}\right)+\pr\left(\tau_1>\frac{t}{2}\right)\\
	&\leq\esp\left[\pr\left(\tau_2>\frac{t}{2}\ \bigg|\	|\zeta_{\tau_1}|\right)\mathbb{1}_{\tau_1\leq\frac{t}{2}}\right]+ C\rho_0\frac{e^{-c\frac{t}{2}}}{\frac{t}{2}}\\
	&\leq\esp\left[C_q e^{-\tilde{c}\frac{t}{2}}|	\zeta_{\tau_1}|^{q}\mathbb{1}_{\tau_1\leq\frac{t}{2}}\right]+ 2C\rho_0\frac{e^{-c\frac{t}{2}}}{t}\\
	&\leq C_q e^{-\tilde{c}\frac{t}{2}}\esp\left[\left|	\zeta_{\tau_1}\right|^{q}\right]+ 2C\rho_0\frac{e^{-c\frac{t}{2}}}{t}.
\end{align*}

As $\rho_0\leq \rho_0+\sqrt{|\zeta_0|}$ with $\rho_0+\sqrt{|\zeta_0|}$ equivalent to $d_{cc}(\bbb_0,\bbb'_0)$, using Proposition \ref{AireBalayéeGnl}, we obtain the expected inequality.
\end{proof}

\section{Applications to gradients estimates}
In this section, we show gradient inequalities involving the heat semi group $(P_t)_t$, that is the semi-group with infinitesimal generator $\frac{1}{2}L$, the subLaplacian operator. For a function $f$ on $E_k$ and $g\in E_k$, we consider the norm of the gradient: \begin{equation*}||\nabla_{\mathcal{H}}f(g)||_{\mathcal{H}}:=\sqrt{(\bar{X}f)^2(g)+(\bar{Y}f)^2(g)}.\end{equation*} We recall the definition of an upper gradient. We say that a function $u$ is an upper gradient of $f$ in the sense that for every horizontal curve $\gamma:[0,T]\rightarrow E_k$ parametrized with the arc-length, we have:
\begin{equation*}
   |f(\gamma(0))-f(\gamma(t))|\leq \int_0^tu(\gamma(s))ds.
\end{equation*}
As $E_k$ has a left-invariant subRiemannian structure, it is a regular subRiemannian manifold as described in~\cite{HuangSublaplacian} and so, we get, for all $u$ upper gradient of $f$:
\begin{equation}
    ||\nabla_{\mathcal{H}}f(g)||_{\mathcal{H}}\leq u(g) \text{ a.e. in }E_k.
\end{equation}
See~\cite{HuangSublaplacian,SobolevMetPoincarre} for some proofs.
In particular, for $f$ Lipschitz, an upper gradient will be given by the gradient length $|\nabla f|(g):=\lim\limits_{r\downarrow 0}\sup\limits_{\substack{
g\neq g'\\
d_{cc}(g,g')<r}
}\left\lvert\frac{f(g)-f(g')}{d_{cc(g,g')}}\right\rvert$ (see~\cite{kuwada,SobolevMetPoincarre}). 
Thus we have:
\begin{equation}\label{Uppergradient}
    ||\nabla_{\mathcal{H}}f(g)||_{\mathcal{H}}\leq |\nabla f|(g) \text{ a.e. in }E_k.
\end{equation}

Then, we can use the coupling rate of Theorem \ref{Succes1}, to obtain the gradient inequality of Corollary \ref{grad}.

\begin{proof}[Proof of Corollary \ref{grad}]
Let $g,g'\in SU(2)$. We consider $(\bbb_t,\bbb'_t)_t$ the coupling constructed in Theorem \ref{Succes1} starting from $(g,g')$ and $\tau$ its first coupling time. As $SU(2)$ is compact and $f$ continuous, then it is bounded and we get:
\begin{align}\label{inegalité}
    |P_tf(g)-P_tf(g')|&=|\esp[f(\bbb_t)-f(\bbb_t')]| \notag\\ 
    &\leq\esp[|f(\bbb_t)-f(\bbb_t')|]= \esp[|f(\bbb_t)-f(\bbb_t')|\mathbb{1}_{{\tau}>t}]\notag\\
    &\leq 2||f||_{\infty}\pr(\tau>t)\\
    &\leq 2||f||_{\infty}Ce^{-ct}d_{cc}(g,g').\notag
\end{align}
In particular, $g\mapsto P_tf(g)$ is Lipschitz and $|\nabla P_tf|(g)\leq 2||f||_{\infty}Ce^{-ct}$. We just use (\ref{Uppergradient}) to obtain: $||\nabla_{\mathcal{H}}P_tf(g)||_{\mathcal{H}}\leq 2||f||_{\infty}Ce^{-ct}$ a.e.\\
If $f$ is harmonic on all $SU(2)$, we have $P_tf=f$ and so $||\nabla_{\mathcal{H}}f(g)||_{\mathcal{H}}\leq2||f||_{\infty}C_1e^{-C_2t}$ for all $t$ a.e. Then, letting $t$ tend to $+\infty$, we obtain $||\nabla_{\mathcal{H}}f(g)||_{\mathcal{H}}=0 $  and so $\bar{X}f=\bar{Y}f=0$ a.e. Using the Lie bracket generating property of $\mathcal{H}=\text{Span}\langle \bar{X},\bar{Y}\rangle$, we get $f$ constant a.e. and, by continuity, $f$ is constant on $SU(2)$.
\end{proof}
The computations to prove Corollary \ref{gradSL} are quite similar:
\begin{proof}[Proof of Corollary \ref{gradSL}]
We suppose that $g=(x,z)$, $g'=(x',z')\in SL(2,\mathbb{R})$ with $x=x'$. The same way as for inequality (\ref{inegalité}), using Theorem \ref{Succes2}, we obtain $|P_tf(g)-P_tf(g')|\leq 2||f||_{\infty}C_q e^{-ct}d_{cc}(g,g')^{2q}$.\\
Moreover, if $f$ is harmonic, $P_tf=f$ and so $|f(g)-f(g')|\leq 2||f||_{\infty}C_q e^{-ct}d_{cc}(g,g')^{2q}$. Taking $t\to 0$, we get $f(g)=f(g')$. Thus, $z\mapsto f(x,z)$ is constant for all $x\in\mathbf{H}^2$.
\end{proof}
	\bibliographystyle{plain}
	\bibliography{Bibliographie}
\end{document}